\documentclass[journal]{IEEEtran}
\usepackage{tikz}

\usepackage{dsfont}
\usepackage{amsmath}
\usepackage{amsthm}
\usepackage{amssymb}
\usepackage{tikz}
\usepackage{subfig}
\usepackage{cite}
\usepackage{hhline}

\usepackage{hyperref}

\newtheorem{theorem}{Theorem}
\newtheorem{remark}[theorem]{Remark}

\newtheorem{definition}[theorem]{Definition}
\newtheorem{algorithm}[theorem]{Algorithm}

\usepackage{caption}

\newcommand{\marking}{0}
\ifnum \marking=1
\definecolor{myMarkingColor}{HTML}{2160ad}
\newcommand{\myMarkingColor}{myMarkingColor}
\newcommand{\myMarkingColorDark}{myMarkingColor!50!black}
\fi
\ifnum \marking=0
\definecolor{myMarkingColor}{rgb}{0,0,0}
\newcommand{\myMarkingColor}{black!60!white}
\newcommand{\myMarkingColorDark}{black}
\fi

\DeclareCaptionFont{myMarkingColor}{\color{myMarkingColor}}

\begin{document}

\title{Resolution Guarantees in Electrical Impedance Tomography}

\author{Bastian~Harrach,
        Marcel~Ullrich
\thanks{B. Harrach (birth name: Bastian Gebauer) is with the Department of Mathematics, University of Stuttgart, Germany e-mail: bastian.harrach@math.uni-stuttgart.de}
\thanks{M. Ullrich is with the Department of Mathematics, University of Stuttgart, Germany e-mail: marcel.ullrich@mathematik.uni-stuttgart.de}}

\IEEEaftertitletext{\vspace*{-5.5cm}%
\fbox{\centering \begin{minipage}{17.5cm}\centering
This is the author's version of an article that has been published in \emph{IEEE Trans. Med. Imaging} \textbf{34}(7), 1513--1521, 2015.\\
Changes were made to this version by the publisher prior to publication.\\
The final version of record is available at \url{http://dx.doi.org/10.1109/TMI.2015.2404133}
\end{minipage}}\vspace{4.5cm}%
}

\maketitle

\begin{abstract}
Electrical impedance tomography (EIT) uses current-voltage measurements on the surface of an imaging subject to
detect conductivity changes or anomalies. EIT is a promising new technique with great potential
in medical imaging and non-destructive testing. However, in many applications, EIT suffers from inconsistent reliability due to its enormous sensitivity to modeling and measurement errors. 

In this work we show that rigorous resolution guarantees are possible within a realistic EIT measurement setting including systematic and random errors. We derive a constructive criterion to decide whether a desired resolution can be achieved in a given measurement setup. Our result covers the detection of anomalies of a known minimal contrast
using noisy measurements on a number of electrodes attached to a subject with imprecisely known background conductivity.
\end{abstract}


\section{Introduction}

{\color{myMarkingColor}\IEEEPARstart{E}{lectrical} impedance tomography (EIT) is an imaging technique that uses current-voltage measurements on
the surface of a conductive subject to image its inner conductivity distribution. From this conductivity image, one can extract information about
the physiological composition of the subject. An upcoming application of EIT is lung monitoring. Since an inflated lung has a lower specific
conductivity than surrounding body tissues, this leads to a visible contrast in the EIT image.
Another promising application which we will focus on in this work, is the detection of anomalies (aka inclusions) where the conductivity significantly differs from an expected background value.
There are several relevant practical scenarios, e.g.\ the detection of tumors or hemorrhages in surrounding homogeneous tissue that has a certain conductivity
contrast.

For a further overview of practical applications of EIT appearing in the fields of medical imaging and material testing of industrial or building materials, cf. e.g., 
%
\cite{barber1984applied,wexler1985impedance,newell1988electric,metherall1996three,cheney1999electrical,borcea2002electrical,borcea2003addendum,lionheart2003eit,holder2004electrical,bayford2006bioimpedance,choi2007reconstruction,halter2008broadband,moura2010dynamic,adler2011electrical,martinsen2011bioimpedance}.
}

The reconstruction process in EIT suffers from the fundamental ill-posedness of the underlying mathematical inverse problem which leads to an enormous 
sensitivity to modeling and measurement errors. Due to these inherent instability issues, 
high resolution EIT imaging remains an extremely challenging topic.
{\color{myMarkingColor}However, several applications would already greatly benefit from low resolution EIT images, e.g. in the field of the aforementioned tumor or hemorrhage detection. For these applications, fast and low-cost monitoring techniques have to be developed
in order to decide which patients should undergo more extensive diagnosis.
For this task, the main concern seems to be the reliability of EIT images.}

The goal of this work is to show that rigorous resolution guarantees are possible within a realistic EIT measurement setting including systematic and random errors. 
Consider a measurement setting as in figure \ref{fig:resolution_numerics} where voltage-current measurements are taken on a number of electrodes attached to the boundary of an imaging domain $\Omega$. The aim is to detect whether the domain contains one or several anomalies where the conductivity differs from some normal background range.

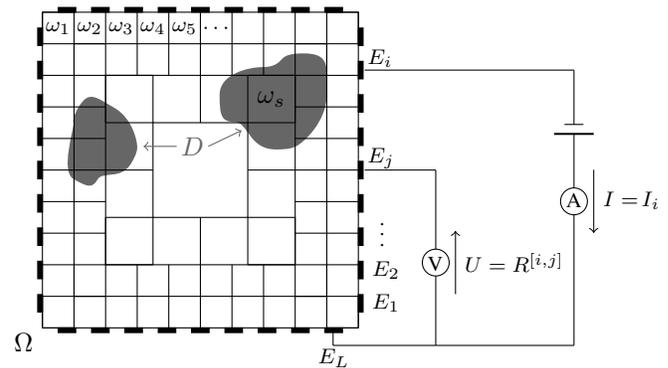
\begin{figure}[!t]
\begin{center}
\begin{tikzpicture}[scale=2.1]

\newcommand{\myL}{36};
\newcommand{\myLdiv}{9};
\newcommand{\myp}{0.5};

\newcommand{\myb}{((8*\myp)/\myL)};
\newcommand{\myh}{((2*(1-\myp))/(\myL/4+1))};

\begin{scope}

\fill[color=black!60!white][scale=1/11.5, shift={(-0.5,5.5)}]
(1.9,0)
.. controls (1.9,-1.5) and (2,-2) .. (3,-2)
.. controls (4,-2) and (4,-4) .. (6,-4)
.. controls (7,-4) and (7.25,-3.875) .. (7.5,-3.75)
.. controls (8,-3.5) and (9.7,-2) .. (9.7,0)
.. controls (9.7,2) and (10,3) .. (8,3)
.. controls (7.25,3) and (6.75,2.9) .. (6.5,2.65)
.. controls (6.25,2.4) and (6,2) .. (5,2)
.. controls (4,2) and (1.9,1.5) .. (1.9,0);
\end{scope}

\begin{scope}
\fill[color=\myMarkingColor][scale=1/15, shift={(-13,-2.5)}, rotate=60]
(1.9,0)
.. controls (1.9,-1.5) and (2,-2) .. (3,-2)
.. controls (4,-2) and (4,-4) .. (6,-4)
.. controls (7,-4) and (7.25,-3.875) .. (7.5,-3.75)
.. controls (8,-3.5) and (9.7,-2) .. (9.7,0)
.. controls (9.7,2) and (10,3) .. (8,3)
.. controls (7.25,3) and (6.75,2.9) .. (6.5,2.65)
.. controls (6.25,2.4) and (1.9,1.5) .. (1.9,0)
;
\end{scope}

\foreach \x in {1,2,...,\myLdiv}{
\draw (1,{-1+\myh+(\x-1)*(\myb+\myh)}) rectangle (1.03,{-1+\myh+(\x-1)*(\myb+\myh)+\myb});
\fill (1,{-1+\myh+(\x-1)*(\myb+\myh)}) rectangle (1.03,{-1+\myh+(\x-1)*(\myb+\myh)+\myb});}
\draw [rotate=-90] (1.025+0.05,{-1+\myh+(9-1)*(\myb+\myh)+\myb/2}) node [anchor=north] {\footnotesize $E_L$};

\begin{scope}[shift={(0.09,0)}]
\draw (1.5-0.1,{-1+\myh+(-0.5-1)*(\myb+\myh)+\myb/2+0.05}) -- (1.5-0.75,{-1+\myh+(-0.5-1)*(\myb+\myh)+\myb/2+0.05});
\draw (1.5-0.75,{-1+\myh+(-0.5-1)*(\myb+\myh)+\myb/2+0.05}) -- (1.5-0.75,-1);
\draw (1.5-0.1,{-1+\myh+(5-1)*(\myb+\myh)+\myb/2}) -- (1.5-0.1,{-1+\myh+(-0.5-1)*(\myb+\myh)+\myb/2+0.05});
\draw (1-0.05,{-1+\myh+(5-1)*(\myb+\myh)+\myb/2}) -- (1.5-0.1,{-1+\myh+(5-1)*(\myb+\myh)+\myb/2});

\draw (1.6-0.075,-1+2*1.12/9+1*0.11+0.055) node [anchor=west] {\footnotesize $U=R^{[i,j]}$};
\filldraw [fill=white] (1.5-0.1,-1+2*1.12/9+1*0.11+0.055) circle (0.085);
\draw (1.6-0.095,-1+2*1.12/9+1*0.11+0.045) node [anchor=east] {\footnotesize V};
\draw [->] (1.6-0.075,-1+2*1.12/9+1*0.11+0.055-0.2) -- (1.6-0.075,-1+2*1.12/9+1*0.11+0.055+0.2);

\draw (1-0.05,{-1+\myh+(8-1)*(\myb+\myh)+\myb/2}) -- (1.5+0.775,{-1+\myh+(8-1)*(\myb+\myh)+\myb/2});
\draw (1.5+0.775,{-1+\myh+(8-1)*(\myb+\myh)+\myb/2}) -- (1.5+0.775,{-1+\myh+(-0.5-1)*(\myb+\myh)+\myb/2+0.05+1.4});
\draw (1.5+0.775-0.06,{-1+\myh+(-0.5-1)*(\myb+\myh)+\myb/2+0.05+1.4}) -- (1.5+0.775+0.06,{-1+\myh+(-0.5-1)*(\myb+\myh)+\myb/2+0.05+1.4});
\draw [line width=0.75pt](1.5+0.775-0.125,{-1+\myh+(-0.5-1)*(\myb+\myh)+\myb/2+0.05+1.34}) -- (1.5+0.775+0.125,{-1+\myh+(-0.5-1)*(\myb+\myh)+\myb/2+0.05+1.34});
\draw (1.5+0.775,{-1+\myh+(-0.5-1)*(\myb+\myh)+\myb/2+0.05+1.34}) -- (1.5+0.775,{-1+\myh+(-0.5-1)*(\myb+\myh)+\myb/2+0.05});
\draw (1.5-0.1,{-1+\myh+(-0.5-1)*(\myb+\myh)+\myb/2+0.05}) -- (1.5+0.775,{-1+\myh+(-0.5-1)*(\myb+\myh)+\myb/2+0.05});

\draw (1.5+0.4+0.375+0.125,-1+6*1.12/9+5*0.11-0.055+0.1-0.535) node [anchor=west] {\footnotesize $I=${\color{myMarkingColor}\,$I_i$}};
\filldraw [fill=white] (1.5+0.4+0.375,-1+6*1.12/9+5*0.11-0.055+0.11-0.55) circle (0.085);
\draw (1.6+0.423+0.36,-1+6*1.12/9+5*0.11-0.055+0.1-0.535) node [anchor=east] {\footnotesize A};
\draw [<-] (1.5+0.4+0.375+0.125,-1+6*1.12/9+5*0.11-0.055+0.11-0.55-0.2) -- (1.5+0.4+0.375+0.125,-1+6*1.12/9+5*0.11-0.055+0.11-0.55+0.2);
\end{scope}

\foreach \x in {1,2}{
\draw              (1.03,{-1+\myh+(\x-1)*(\myb+\myh)+\myb/2}) node [anchor=west] {\footnotesize $E_\x$};}
\foreach \x in {3.3}{
\draw              (1.075,{-1+\myh+(\x-1)*(\myb+\myh)+\myb/2}) node [anchor=west] {\footnotesize $\vdots$};}
\foreach \x in {5}{
\draw              (1,{-1+\myh+(\x-1)*(\myb+\myh)+\myb/2+0.075}) node [anchor=west] {\footnotesize $E_j$};}
\foreach \x in {8}{
\draw              (1,{-1+\myh+(\x-1)*(\myb+\myh)+\myb/2+0.075}) node [anchor=west] {\footnotesize $E_{i}$};}

\foreach \x in {1,2,...,\myLdiv}{
\draw [rotate=90]  (1,{-1+\myh+(\x-1)*(\myb+\myh)}) rectangle (1.03,{-1+\myh+(\x-1)*(\myb+\myh)+\myb});
\fill [rotate=90]  (1,{-1+\myh+(\x-1)*(\myb+\myh)}) rectangle (1.03,{-1+\myh+(\x-1)*(\myb+\myh)+\myb});
}
\foreach \x in {1,2,...,\myLdiv}{
\draw [rotate=180] (1,{-1+\myh+(\x-1)*(\myb+\myh)}) rectangle (1.03,{-1+\myh+(\x-1)*(\myb+\myh)+\myb});
\fill [rotate=180] (1,{-1+\myh+(\x-1)*(\myb+\myh)}) rectangle (1.03,{-1+\myh+(\x-1)*(\myb+\myh)+\myb});
}
\foreach \x in {1,2,...,\myLdiv}{
\draw [rotate=-90] (1,{-1+\myh+(\x-1)*(\myb+\myh)}) rectangle (1.03,{-1+\myh+(\x-1)*(\myb+\myh)+\myb});
\fill [rotate=-90] (1,{-1+\myh+(\x-1)*(\myb+\myh)}) rectangle (1.03,{-1+\myh+(\x-1)*(\myb+\myh)+\myb});
}

\draw (-1,-1) rectangle (1,1);
\draw (-0.8,-0.8) rectangle (0.8,0.8);
\draw (-0.6,-0.6) rectangle (0.6,0.6);
\draw (-0.3,-0.3) rectangle (0.3,0.3);

\foreach \x in {-1,-0.8,...,1}{
\draw (-1,\x) -- (-0.8,\x);
\draw (1,\x) -- (0.8,\x);
\draw (\x,-1) -- (\x,-0.8);
\draw (\x,1) -- (\x,0.8);}

\foreach \x in {1,2,...,5}{
\draw (-1.1+0.2*\x,0.9) node {\footnotesize $\omega_\x$};}
\draw (-1.09+0.2*6,0.9) node {\footnotesize $\cdots$};
\draw (0.45,0.45) node {$\omega_{{\color{\myMarkingColorDark}s}}$};

\foreach \x in {-0.8,-0.6,...,0.8}{
\draw (-0.8,\x) -- (-0.6,\x);
\draw (0.8,\x) -- (0.6,\x);
\draw (\x,-0.8) -- (\x,-0.6);
\draw (\x,0.8) -- (\x,0.6);}

\foreach \x in {-0.6,-0.3,...,0.6}{
\draw (-0.6,\x) -- (-0.3,\x);
\draw (0.6,\x) -- (0.3,\x);
\draw (\x,-0.6) -- (\x,-0.3);
\draw (\x,0.6) -- (\x,0.3);}

\draw[color=\myMarkingColor] (-0.05,0.15) node {$D$};
\draw[color=\myMarkingColor][->] (0.05,0.15) -- (0.26,0.25);
\draw[color=\myMarkingColor][->] (-0.15,0.15) -- (-0.36,0.15);
\draw (-1,-0.975) node [anchor=north east] {$\Omega$};

\end{tikzpicture}
\caption{\color{myMarkingColor}Measurement setting with inclusions $D$ occupying a subset of a subject
$\Omega$ that is decomposed into a partition of subsets $\omega_1,\omega_2,\ldots\subseteq\Omega$.
Driving a current 
through the $i$-th and the $L$-th electrode, we measure the corresponding voltage $R^{[i,j]}$ (in mV per applied mA) between the $j$-th and the $L$-th electrode. Repeating 
this for all $i$ and $j$ we obtain the measurement matrix $R=(R^{[i,j]})_{i,j=1,\ldots, L-1}\in \mathbb{R}^{(L-1)\times (L-1)}$.}
\label{fig:resolution_numerics}
\end{center}
\end{figure}

We describe a desired \emph{resolution} by a partition of $\Omega$ into disjoint subsets $\omega_1,\omega_2,\ldots\subseteq\Omega$. 
We say that a \emph{resolution guarantee} holds if the measured data contains enough information to 
\begin{enumerate}
\item[(a)] correctly mark every element $\omega_{{\color{myMarkingColor}s}}$ that is completely covered by an anomaly{\color{myMarkingColor},}
\item[(b)] correctly mark no element, if there is no anomaly {\color{myMarkingColor}at all.}
\end{enumerate}
In other words, a resolution guarantee ensures that no false positives are detected in the anomaly-free case, 
and no false negatives are detected in the case of inclusions over a certain size.
{\color{myMarkingColor}Let us stress that in this work we aim to characterize the resolution 
up to which an anomaly can be detected. Assumptions (a) and (b) do not guarantee that the shape of a 
detected anomaly can be correctly determined up to a certain resolution. In that sense, the subject of this work might be called a (resolution-based) detection guarantee.}

Whether a certain desired resolution can be guaranteed will depend on a number of facts, including the number and position 
of electrodes, the measurement pattern, the inclusion contrast, and modeling and measurement errors.
The aim of this work is to show that resolution guarantees are possible in realistic settings, and to derive a criterion to evaluate whether a desired resolution can be guaranteed. 
We also describe a simple reconstruction algorithm that implements (a) and (b) above.

Let us comment on the vast literature on identfiability in EIT. In the last decades, great theoretical progress has been made on
the question whether two arbitrary conductivities can be distinguished by idealized noise-free and continuous measurements (the \emph{Calder\'on-Problem}
\cite{calderon1980inverse,calderon2006inverse})
. We refer to the seminal works
\cite{kohn1984determining,kohn1985determining,nachman1996global,astala2006calderon}
, the overview
\cite{uhlmann2008commentary}
%
and the recent breakthroughs for partial boundary data 
\cite{imanuvilov2011determination,kenig2013recent}%
. The distinguishability of conductivities from finite precision data has been studied in the works of Bates, {\color{myMarkingColor}Gen\c{c}er,} Gisser, {\color{myMarkingColor}Ider,} Isaacson, {\color{myMarkingColor}Kuzuoglu, Lionheart,} Newell, Seagar, {\color{myMarkingColor}Paulson, Pidcock} and Yeo
{\color{myMarkingColor}\cite{seagar1984full,seagar1985full,isaacson1986distinguishability,gisser1987current,gisser1990electric,paulson1993optimal,gencer1994electrical}. Also, let us refer to the works of Kolehmainen, Lassas, Nissinen, Ola and Kaipio \cite{kolehmainen2008electrical,nissinen2011compensation} regarding uncertainties in the subject's shape and
electrode's contact impedances.}

Several reconstruction methods have been proposed for anomaly or inclusion detection problems,
cf., e.g., Potthast \cite{potthast2006survey} for an overview. Arguably, the most prominent inclusion detection method is the Factorization Method (FM) of Kirsch, Br\"uhl and Hanke \cite{Kir98,Bru00,Bru01}, see
\cite{hanke2003recent,hyvonen2004complete,kirsch2005factorization,gebauer2006factorization,gebauer2007factorization,hyvonen2007numerical,nachman2007imaging,gebauer2008localized,hanke2008factorization,kirsch2008factorization,lechleiter2008factorization,hakula2009computation,harrach2009detecting,schmitt2009factorization,harrach2010factorization,adler2011electrical,schmitt2011factorization}
for the devolopment of the FM in the field of EIT and
\cite{harrach2013recent}
for a recent overview. Notably, in the overview
\cite{harrach2013recent}
, the FM is formulated on the basis of monotonicity-based arguments, and the recent result \cite{harrach2013monotonicity} indicates that, for EIT,
the FM can be outperformed by monotonicity-based methods first formulated by Tamburrino and Rubinacci in \cite{Tamburrino02,Tamburrino06}. 

The main new idea of this work is to obtain resolution guarantees for realistic settings by treating 
worst-case scenarios with monotonicity-based ideas. To the knowledge of the authors, the results derived herein 
are the firsts to rigorously quantify the achievable resolution in the case of realistic electrode measurements in a setting
with imprecisely known background conductivity, contact impedances and measurement noise. {\color{myMarkingColor}We believe that this will be useful for designing reliable EIT systems. Our results may be used to determine whether a 
desired resolution can be achieved and to quantify the required measurement accuracy.
Moreover, our results could be the basis of optimization strategies regarding the resolution, or the number and positions of electrodes and the driving patterns.}

The paper is organized as follows. A realistically {\color{myMarkingColor}modeled} measurement setting including systematic and random errors is introduced in section \ref{sec:CEM}.
Section \ref{sec:monotonicity_based_methods} presents a monotonicity relation and motivates how this relation can be used to design inclusion detection methods. In section \ref{sec:RGs}, we introduce the concept of a rigorous resolution guarantee and show how to verify such
guarantees by a simple test. We also derive fast {\color{myMarkingColor}linearized} versions of our tests that allow faster verifications at the price of underestimating the achievable resolution. Section \ref{sec:numerical_results} presents some numerical results for the verification of resolution guarantees of section \ref{sec:RGs}.
Section \ref{sec:conclusion} contains some concluding remarks.

\section{The setting}\label{sec:CEM}

The current-voltage measurements can be modeled by the complete electrode
model (CEM) as follows (cf. \cite{Som92}). Let $\Omega\subseteq\mathbb{R}^n$ be a bounded domain with piecewise smooth boundary representing the conductive object and let $\sigma:\ \Omega\to \mathbb{R}$ be the real valued conductivity distribution inside $\Omega$. We assume that $\sigma$ and $1/\sigma$ are positive and bounded functions.

Electrodes are attached to the boundary of the object 
as in figure \ref{fig:resolution_numerics}.
Their location is denoted with $E_1,E_2,\ldots,E_L\subseteq\partial\Omega$, and their
contact impedances are denoted by a vector with positive entries
\begin{equation*}
z:=\left(z^{[1]},\ldots,z^{[L]}\right)\in\mathbb{R}^L.
\end{equation*}
{\color{myMarkingColor}The electrodes are assumed to be perfectly conductive.} 

For each $i\in\lbrace 1,2,\ldots,L-1\rbrace$, we drive a current {\color{myMarkingColor}$I_i$} with strength
$1$ {\color{myMarkingColor}$m$A} through the $i$-th electrode while keeping the $L$-th electrode grounded
and all other electrodes insulated (so that the current {\color{myMarkingColor}flows} through the grounded $L$-th electrode). Then the potential $u_i$ inside $\Omega$
and the potentials $U_i=\left(U_i^{[1]},\ldots,U_i^{[L]}\right)$ on the electrodes fulfill
\begin{equation*}\label{eq:CEM_beginning}
 \nabla\cdot\sigma\nabla u_i=0\quad \textrm{in}\ \Omega,
\end{equation*}
with boundary conditions
\setlength{\arraycolsep}{0.0em}
\begin{eqnarray*}\label{eq:CEM_ending}
 \int_{E_l}\sigma\left(\nabla u_i\right)\cdot\nu\,\mathrm{d}S&{}={}&
\delta_{l,i}-\delta_{l,L},\\
 u_i|_{E_l}+z^{[l]}\sigma\left(\nabla u_i\right)\cdot\nu |_{E_l} &{}={}&
\text{const.}=:
 U_i^{[l]}
\end{eqnarray*}
\setlength{\arraycolsep}{5pt}
for $l\in\lbrace 1,2,\ldots,L\rbrace$,
\begin{equation*}
  \sigma\left(\nabla u_i\right)\cdot\nu = 0\quad\textrm{on}\quad \partial\Omega\setminus\bigcup_{l=1}^L E_l,
\end{equation*}
and $U_i^{[L]}=0$. $\nu$ is the outer normal on the boundary of $\Omega$.
 
For each injected current, we measure the voltage on $E_1$,\ldots,$E_{L-1}$ 
against the grounded $L$-th electrode. We thus collect a matrix of 
measurements 
\begin{equation}\label{eq:definition_of_R}
R(\sigma,z):=\left(R^{[{\color{myMarkingColor}i,j}]}(\sigma,z)\right)_{{\color{myMarkingColor}i,j}=1}^{L-1}:=\left(U_i^{[{\color{myMarkingColor}j}]}\right)_{{\color{myMarkingColor}i,j}=1}^{L-1}\in\mathbb{R}^{\color{\myMarkingColor}(L-1)^2}.
\end{equation}
The matrix $R(\sigma,z)$ is easily shown to be symmetric.


We consider anomaly detecting problems where we try to detect
regions (the so-called inclusions) in $\Omega$ where the conductivity differs from a normal background range. To allow for {\color{myMarkingColor}modeling} and
measurement errors in this context, we make the
following setting assumptions:

\begin{enumerate}
 \item[(a)] {\bf Conductivity distribution $\sigma(x)$:} The true conductivity distribution is given by an inclusion conductivity $\sigma_D(x)$ inside an inclusion $D$ and
by a background conductivity $\sigma_B(x)$ inside $\Omega\setminus D$, i.e.
$$\sigma(x)=\begin{cases}\sigma_D(x),&x\in D,\\\sigma_B(x),&x\in\Omega\setminus D.\end{cases}$$
 \item[(b)] {\bf Background error $\epsilon\geq0$:} The background conductivity approximately agrees with a known positive constant $\sigma_0>0$, 
 $$\left|\sigma_B(x)-\sigma_0\right|\leq\epsilon \quad \forall x\in\Omega\setminus D.$$
 \item[(c)] {\bf Inclusion conductivity contrast $c>0$:} 
We assume that we know a lower bound on the inclusion contrast, i.e.,
that we know $c>0$ with either
 \begin{enumerate}
  \item[(i)] $\sigma_D(x)-\sigma_0\geq c\quad \forall x\in D$,
  \item[(ii)] $\sigma_0-\sigma_D(x)\geq c\quad \forall x\in D$.
 \end{enumerate}
 \item[(d)] {\bf Contact impedances error $\gamma\geq0$:}
We assume that we approximately know the contact impe\-dances vector $z$,
i.e. that we know $z_0\in \mathbb{R}^L$ with
$$\left|z^{[l]}-z_0^{[l]}\right|\leq\gamma\quad \forall l\in\lbrace 1,2,\ldots,L-1\rbrace.$$
 \item[(e)] {\bf Measurement noise $\delta\geq0$:} We assume that
we can measure the voltages $R(\sigma,z)$ up to a noise level $\delta>0$, i.e., that we are given $R_\delta\in\mathbb{R}^{(L-1)\times (L-1)}$ with
$$\left\|R(\sigma,z)-R_\delta\right\|\leq\delta.$$
Possibly replacing $R_\delta$ by its symmetric
part, we can assume that $R_\delta$ is symmetric. 
\end{enumerate}

\section{Monotonicity}\label{sec:monotonicity_based_methods}

Our results are based on the following monotonicity relations that extend
results of Gisser, Ikehata, Isaacson, Kang, Newell, Rubinacci, 
Seo, Sheen, and Tamburrino  \cite{gisser1990electric, ikehata1998size,Kan97, Tamburrino02}.

\begin{theorem}\label{th:monotonicity_relation}
 For $i\in\lbrace 1,2\rbrace$, let $\sigma_i:\ \Omega\to \mathbb{R}$ be a conductivity distribution and $z_i\in\mathbb{R}^L$ be a contact impedances vector.
 Then
\begin{equation}\label{eq:monotony_relation}
 \sigma_1\leq\sigma_2, z_1\geq z_2 
\quad \text{ implies } \quad
 R\left(\sigma_1,z_1\right)\geq R\left(\sigma_2,z_2\right).
\end{equation}
The inequalities on the left side of the implication are meant pointwise.
The inequality on the ride side is to be understood in the sense of matrix definiteness.
\end{theorem}

\begin{proof}
This follows from the more general theorem~\ref{th:monotonicity_inequality}
below.
\end{proof}

Theorem \ref{th:monotonicity_relation} yields monotonictiy-based inclusion detection methods, cf.~\cite{Tamburrino02}. To present the main idea, consider the simple example
where $\sigma=1+\chi_D$, where $\chi_D$ is the characteristic function on $D$,
and the contact impedances vector $z\in\mathbb{R}^L$ is known exactly.

For a small ball $B\subseteq\Omega$ we define a test conductivity
$\tau_B=1+\chi_B$. From the monotonicity relation of theorem \ref{th:monotonicity_relation} we have that
\begin{equation*}
 B\subseteq D\quad \text{implies} \quad R(\tau_B,z)\geq R(\sigma,z). 
\end{equation*}
Hence, the union of all test balls $B$ fulfilling $R(\tau_B,z)\geq R(\sigma,z)$ is an upper bound of the inclusion $D$.

In the recent work \cite{harrach2013monotonicity}, the authors showed that,
for continuous boundary data, monotonicity methods are actually capable of reconstructing the exact shape $D$ under rather general assumptions.
Moreover, \cite{harrach2013monotonicity} shows how to replace the monotonicity tests by fast {\color{myMarkingColor}linearized} versions without {\color{myMarkingColor}losing} shape information, see also \cite{harrach2010exact}.

We cannot expect exact shape reconstruction in settings 
with a finite number of electrodes and imprecisely known contact impedances and background conductivities. However, monotonicity-based arguments will allow us
to characterize the achievable resolution in such realistic settings. For this, we formulate a quantitative version of theorem \ref{th:monotonicity_relation}:

\begin{theorem}\label{th:monotonicity_inequality}
For $i\in\lbrace 1,2\rbrace$, let $\sigma_i:\ \Omega\to \mathbb{R}$ be a conductivity distribution and $z_i\in\mathbb{R}^L$ be a contact impedances vector. Given $w\in \mathbb{R}^{L-1}$, let $\left(v_i,V_i\right)$ be the
corresponding potentials resulting from driving a current of $w_j$ through the $j$-th electrode, respectively. (Note that this implies a current flux of $-\sum_{l=1}^L w_l$ through the grounded $L$-th electrode.) Then,
\setlength{\arraycolsep}{0.0em}
\begin{eqnarray*}\label{eq:monotony_inequality}
\lefteqn{
\int_\Omega\left(\sigma_1-\sigma_2\right) \left|\nabla v_2\right|^2 \mathrm{d}x
}
\\
\lefteqn{{+}\:\sum_{l=1}^L\left(\frac{1}{z_1^{[l]}}-\frac{1}{z_2^{[l]}}\right)\int_{E_l}\left(v_2-V_2^{[l]}\right)^2\,\mathrm{d}s}\\
&{}\geq{}& w^T\left(R\left(\sigma_2,z_2\right)-R\left(\sigma_1,z_1\right)\right)w\\
&{}\geq{}& \int_\Omega\frac{\sigma_2}{\sigma_1}\left(\sigma_1-\sigma_2\right)\left|\nabla v_2\right|^2\mathrm{d}x\\
&&{+}\:\sum_{l=1}^L\frac{z_1^{[l]}}{z_2^{[l]}}\left(\frac{1}{z_1^{[l]}}-\frac{1}{z_2^{[l]}}\right)\int_{E_l}\left(v_2-V_2^{[l]}\right)^2\,\mathrm{d}s.
\end{eqnarray*}
\setlength{\arraycolsep}{5pt}
\end{theorem}
\begin{proof}
From the variational formulation of the CEM (cf., e.g., \cite{Som92}),
we obtain for $i,j\in \{1,2\}$,
\setlength{\arraycolsep}{0.0em}
\begin{eqnarray*}
w^T V_j &{}={} & \int_\Omega\sigma_i\nabla v_i\cdot\nabla v_j\mathrm{d}x \\
&&{+}\: \sum_{l=1}^L\frac{1}{z^{[l]}_i}\int_{E_l}\left(v_i-V^{[l]}_i\right)\left(v_j-V^{[l]}_j\right)\mathrm{d}s\\
&{}=:{}& B_{i}((v_i,V_i),(v_j,V_j)).
\end{eqnarray*}
\setlength{\arraycolsep}{5pt}
and, by linearity, we have that
\[
V_j=R(\sigma_j,z_j) w, \quad j\in \{1,2\}.
\]

Hence, it holds that
\setlength{\arraycolsep}{0.0em}
\begin{eqnarray*}
\lefteqn{w^T\left(R\left(\sigma_2,z_2\right)-R\left(\sigma_1,z_1\right)\right)w}\\
  &{}={}&w^T\left(V_2-V_1\right)\\
%
&{}={}&2 B_1\left(\left(v_1,V_1\right),\left(v_2,V_2\right)\right)
- B_2\left(\left(v_2,V_2\right),\left(v_2,V_2\right)\right)\\
 &&{-}\: B_1\left(\left(v_1,V_1\right),\left(v_1,V_1\right)\right)\\
%
%
%
  &{}={}& -\int_\Omega {\sigma_1}\left|\nabla\left(v_1-v_2\right)\right|^2\mathrm{d}x\\
  &&{-}\:\sum_{l=1}^L\frac{1}{z_1^{[l]}}\int\limits_{E_l}\left(\left(v_1-V^{[l]}_1\right)-\left(v_2-V^{[l]}_2\right)\right)^2\mathrm{d}s\\
  &&{+}\:\int_\Omega \left(\sigma_1-\sigma_2\right)\left|\nabla v_2\right|^2\mathrm{d}x\\
  &&{+}\:\sum_{l=1}^L\left(\frac{1}{z_1^{[l]}}-\frac{1}{z_2^{[l]}}\right)\int_{E_l}\left(v_2-V_2^{[l]}\right)^2\,\mathrm{d}s.
\end{eqnarray*}
\setlength{\arraycolsep}{5pt}
Since the first two summands are non-negative, the first inequality of the theorem follows.
 
Interchanging the pairs $(\sigma_1,z_1)$ and $(\sigma_2, z_2)$ and applying
\setlength{\arraycolsep}{0.0em}
\begin{eqnarray*}
 \lefteqn{\sigma_2\left|\nabla\left(v_2-v_1\right)\right|^2+\left(\sigma_1-\sigma_2\right)\left|\nabla v_1\right|^2}\\
 &{}={}&\sigma_1\left|\nabla v_1-\frac{\sigma_2}{\sigma_1}\nabla v_2\right|^2+\frac{\sigma_2}{\sigma_1}\left(\sigma_1-\sigma_2\right)\left|\nabla v_2\right|^2 
\end{eqnarray*}
\setlength{\arraycolsep}{5pt}
and
\setlength{\arraycolsep}{0.0em}
\begin{eqnarray*}
\lefteqn{\frac{1}{z_2^{[l]}}\left(\left(v_2-V_2^{[l]}\right)-\left(v_1-V_1^{[l]}\right)\right)^2}\\
\lefteqn{{+}\:\left(\frac{1}{z_1^{[l]}}-\frac{1}{z_2^{[l]}}\right)\left(v_1-V_1^{[l]}\right)^2}\\
 &{}={}&\frac{1}{z_1^{[l]}}\left(\left(v_1-V_1^{[l]}\right)-\frac{z_1^{[l]}}{z_2^{[l]}}\left(v_2-V_2^{[l]}\right)\right)^2\\
 &&{+}\:\frac{z_1^{[l]}}{z_2^{[l]}}\left(\frac{1}{z_1^{[l]}}-\frac{1}{z_2^{[l]}}\right)\left(v_2-V_2^{[l]}\right)^2
\end{eqnarray*}
\setlength{\arraycolsep}{5pt}
 yields
\setlength{\arraycolsep}{0.0em}
\begin{eqnarray*}
\lefteqn{w^T\left(R\left(\sigma_2,z_2\right)-R\left(\sigma_1,z_1\right)\right)w}\\
&{}={}&\int_{\Omega}\sigma_1\left|\nabla v_1-\frac{\sigma_2}{\sigma_1}\nabla v_2\right|^2\,\mathrm{d}x 
+
\int_{\Omega}\frac{\sigma_2}{\sigma_1}\left(\sigma_1-\sigma_2\right)\left|\nabla v_2\right|^2\,\mathrm{d}x\\
&&{+}\:\sum_{l=1}^L\int\limits_{E_l}\frac{1}{z_1^{[l]}}\left(\left(v_1-V_1^{[l]}\right)-\frac{z_1^{[l]}}{z_2^{[l]}}\left(v_2-V_2^{[l]}\right)\right)^2\mathrm{d}s\\
&&{+}\:\sum_{l=1}^L\int_{E_l}\frac{z_1^{[l]}}{z_2^{[l]}}\left(\frac{1}{z_1^{[l]}}-\frac{1}{z_2^{[l]}}\right)\left(v_2-V_2^{[l]}\right)^2\,\mathrm{d}s.
\end{eqnarray*}
\setlength{\arraycolsep}{5pt}
Since the last two summands are non-negative, the second inequality of the theorem is proven.
\end{proof}

\section{Resolution guarantees}\label{sec:RGs}

In this section we introduce the concept of rigorous resolution guarantees
and show how to verify such guarantees by a simple test. We consider the setting described in 
section \ref{sec:CEM}.

\begin{definition}
An inclusion detection method that yields a reconstruction $D_R$ to the true inclusion $D$
is said to fulfill a \textbf{resolution guarantee} with respect to a partition $(\omega_{{\color{myMarkingColor}s}})_{{{\color{myMarkingColor}s}}=1}^N$ 
if
%
%
\begin{itemize}
 \item[(i)] $\omega_{{\color{myMarkingColor}s}}\subseteq D$ implies $\omega_{{\color{myMarkingColor}s}}\subseteq D_R$ for ${{\color{myMarkingColor}s}}\in\lbrace 1,2,\ldots,N\rbrace$ {\color{myMarkingColor}(i.e.,
every element that is covered by the inclusion will correctly be marked in the reconstruction)},
and
 \item[(ii)] $D=\emptyset$ implies $D_R=\emptyset$ {\color{myMarkingColor}(i.e., 
if there is no inclusion then no element will be marked in the reconstruction)}.
\end{itemize}
\end{definition}

\begin{figure}[!ht]
\center
\begin{tikzpicture}[scale=2]

\newcommand{\myL}{36};
\newcommand{\myLdiv}{9};
\newcommand{\myp}{0.5};

\newcommand{\myb}{((8*\myp)/\myL)};
\newcommand{\myh}{((2*(1-\myp))/(\myL/4+1))};

\begin{scope}

\fill[color=black!60!white][scale=1/11.5, shift={(-0.5,5.5)}]
(1.9,0)
.. controls (1.9,-1.5) and (2,-2) .. (3,-2)
.. controls (4,-2) and (4,-4) .. (6,-4)
.. controls (7,-4) and (7.25,-3.875) .. (7.5,-3.75)
.. controls (8,-3.5) and (9.7,-2) .. (9.7,0)
.. controls (9.7,2) and (10,3) .. (8,3)
.. controls (7.25,3) and (6.75,2.9) .. (6.5,2.65)
.. controls (6.25,2.4) and (6,2) .. (5,2)
.. controls (4,2) and (1.9,1.5) .. (1.9,0);
\end{scope}

\begin{scope}
\fill[color=\myMarkingColor][scale=1/15, shift={(-13,-2.5)}, rotate=60]
(1.9,0)
.. controls (1.9,-1.5) and (2,-2) .. (3,-2)
.. controls (4,-2) and (4,-4) .. (6,-4)
.. controls (7,-4) and (7.25,-3.875) .. (7.5,-3.75)
.. controls (8,-3.5) and (9.7,-2) .. (9.7,0)
.. controls (9.7,2) and (10,3) .. (8,3)
.. controls (7.25,3) and (6.75,2.9) .. (6.5,2.65)
.. controls (6.25,2.4) and (1.9,1.5) .. (1.9,0)
;
\end{scope}

\foreach \x in {1,2,...,\myLdiv}{
\draw (1,{-1+\myh+(\x-1)*(\myb+\myh)}) rectangle (1.03,{-1+\myh+(\x-1)*(\myb+\myh)+\myb});
\fill (1,{-1+\myh+(\x-1)*(\myb+\myh)}) rectangle (1.03,{-1+\myh+(\x-1)*(\myb+\myh)+\myb});}

\foreach \x in {1,...,6}{
\draw              (1.03,{-1+\myh+(\x-1)*(\myb+\myh)+\myb/2}) node [anchor=west] {\footnotesize $E_\x$};}
\foreach \x in {7.25}{
\draw              (1.075,{-1+\myh+(\x-1)*(\myb+\myh)+\myb/2}) node [anchor=west] {\footnotesize $\vdots$};}

\foreach \x in {1,2,...,\myLdiv}{
\draw [rotate=90]  (1,{-1+\myh+(\x-1)*(\myb+\myh)}) rectangle (1.03,{-1+\myh+(\x-1)*(\myb+\myh)+\myb});
\fill [rotate=90]  (1,{-1+\myh+(\x-1)*(\myb+\myh)}) rectangle (1.03,{-1+\myh+(\x-1)*(\myb+\myh)+\myb});
}
\foreach \x in {1,2,...,\myLdiv}{
\draw [rotate=180] (1,{-1+\myh+(\x-1)*(\myb+\myh)}) rectangle (1.03,{-1+\myh+(\x-1)*(\myb+\myh)+\myb});
\fill [rotate=180] (1,{-1+\myh+(\x-1)*(\myb+\myh)}) rectangle (1.03,{-1+\myh+(\x-1)*(\myb+\myh)+\myb});
}
\foreach \x in {1,2,...,\myLdiv}{
\draw [rotate=-90] (1,{-1+\myh+(\x-1)*(\myb+\myh)}) rectangle (1.03,{-1+\myh+(\x-1)*(\myb+\myh)+\myb});
\fill [rotate=-90] (1,{-1+\myh+(\x-1)*(\myb+\myh)}) rectangle (1.03,{-1+\myh+(\x-1)*(\myb+\myh)+\myb});
}

\draw (-1,-1) rectangle (1,1);
\draw (-0.8,-0.8) rectangle (0.8,0.8);
\draw (-0.6,-0.6) rectangle (0.6,0.6);
\draw (-0.3,-0.3) rectangle (0.3,0.3);

\foreach \x in {-1,-0.8,...,1}{
\draw (-1,\x) -- (-0.8,\x);
\draw (1,\x) -- (0.8,\x);
\draw (\x,-1) -- (\x,-0.8);
\draw (\x,1) -- (\x,0.8);}

\foreach \x in {1,2,...,5}{
\draw (-1.1+0.2*\x,0.9) node {\footnotesize $\omega_\x$};}
\draw (-1.09+0.2*6,0.9) node {\footnotesize $\cdots$};
\draw (0.45,0.45) node {$\omega_{{\color{\myMarkingColorDark}s}}$};

\foreach \x in {-0.8,-0.6,...,0.8}{
\draw (-0.8,\x) -- (-0.6,\x);
\draw (0.8,\x) -- (0.6,\x);
\draw (\x,-0.8) -- (\x,-0.6);
\draw (\x,0.8) -- (\x,0.6);}

\foreach \x in {-0.6,-0.3,...,0.6}{
\draw (-0.6,\x) -- (-0.3,\x);
\draw (0.6,\x) -- (0.3,\x);
\draw (\x,-0.6) -- (\x,-0.3);
\draw (\x,0.6) -- (\x,0.3);}

\draw[color=\myMarkingColor] (-0.05,0.15) node {$D$};
\draw[color=\myMarkingColor][->] (0.05,0.15) -- (0.26,0.25);
\draw[color=\myMarkingColor][->] (-0.15,0.15) -- (-0.36,0.15);
\draw (-1,-0.975) node [anchor=north east] {$\Omega$};

\end{tikzpicture}
\caption{\label{fig:resolution_and_setting} Setting with a sample inclusion and resolution.}
\end{figure}
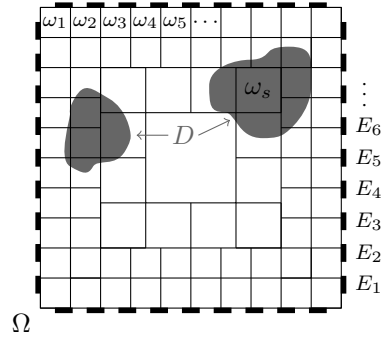

Hence, if a resolution guarantee holds true then no false positives are detected
in the anomaly-free case, and no false negatives are detected in the case of 
inclusions over a certain size. 

Obviously, a resolution guarantee will not hold true for arbitrarily fine partitions. 
The achievable resolution will depend on the number of electrodes, the inclusion's contrast,
the background error, contact impedances error, and the measurement noise, 
cf.\ section~\ref{sec:CEM}(a)-(e).

We will derive a simple test to verify whether a resolution guarantee holds true for a
given partition. To this end, we first consider the case of inclusions that are more conductive than the background. 
The analogous results for less conductive inclusions are summarized in section \ref{sec:less_conductive_inclusions}.
We use the following notations:
\setlength{\arraycolsep}{0.0em}
\begin{eqnarray*}
 {\sigma_B}_\mathrm{min}&{}:={}&\sigma_0-\epsilon,\\
 {\sigma_B}_\mathrm{max}&{}:={}&\sigma_0+\epsilon,\\
 {\sigma_D}_\mathrm{min}&{}:={}&\sigma_0+c,\\
          z_\mathrm{min}&{}:={}&z_0-\gamma(1,1,\ldots,1),\\
          z_\mathrm{max}&{}:={}&z_0+\gamma(1,1,\ldots,1).
\end{eqnarray*}
\setlength{\arraycolsep}{5pt}

\subsection{Verification of resolution guarantees}\label{subsec:A_monotonicity_based_verification}

To verify whether a resolution guarantee holds in a given setting, we will apply 
the following monotonicity-based inclusion detection method. {\color{myMarkingColor}In the following,
we denote the set of eigenvalues of a symmetric square matrix $A$ by $\mathrm{eig}(A)$
and we write $A_1\geq A_2$ (or $A_2\leq A_1$) if the difference $A_1-A_2$ of two symmetric square matrices 
is positive definite, i.e. if $A_1-A_2$ possesses only non-negative eigenvalues.}

\begin{algorithm}\label{algo:nonlinearised_reconstruction_algorithm}
Mark each resolution element $\omega_{{\color{myMarkingColor}s}}$ for which
\begin{equation}\label{eq:test_criterion_algorithm_1_nonlinearised}
 R(\tau_{{\color{myMarkingColor}s}},z_\textrm{max})+\delta\mathrm{Id}\geq R_\delta,
\end{equation}
where
\begin{equation}\label{eq:def_taus}
\tau_{{\color{myMarkingColor}s}}:={\sigma_B}_\mathrm{min}\chi_{\Omega\setminus{\omega_{{\color{myMarkingColor}s}}}}+{\sigma_D}_\mathrm{min}\chi_{\omega_{{\color{myMarkingColor}s}}},\quad {{\color{myMarkingColor}s}}\in\lbrace 1,2,\ldots, N\rbrace.
\end{equation}

Then the reconstruction $D_R$ is given by the union of the marked resolution elements.
\end{algorithm}

\begin{theorem}\label{th:theorem_to_algorithm_nonlinearised}
 The reconstruction of algorithm \ref{algo:nonlinearised_reconstruction_algorithm} fulfils 
the resolution guarantee if 
 \begin{equation}\label{eq:verification_criterion_nonlinearised}
  \mu<-2\delta\leq 0
 \end{equation}
 with
 \begin{equation}\label{eq:worst_case_test}
  \mu:=\max_{{{\color{myMarkingColor}s}}=1}^N\left(\min\left(\mathrm{eig}\left(R(\tau_{{\color{myMarkingColor}s}},z_\mathrm{max})-R({\sigma_B}_\mathrm{max},z_\mathrm{min}))\right)\right)\right).
 \end{equation}
\end{theorem}
\begin{proof}
First, let $\omega_{{\color{myMarkingColor}s}}\subseteq D$. Then, $\tau_{{\color{myMarkingColor}s}}\leq \sigma$ and $z_\mathrm{max}\geq z$.
Theorem \ref{th:monotonicity_relation} implies that 
\[
R(\sigma,z)\leq R(\tau_{{\color{myMarkingColor}s}},z_\mathrm{max}).
\]
Hence, $ R(\tau_{{\color{myMarkingColor}s}},z_\textrm{max})+\delta\mathrm{Id}\geq R_\delta$, so that
$\omega_{{\color{myMarkingColor}s}}$ will be marked by algorithm \ref{algo:nonlinearised_reconstruction_algorithm}.
This shows that part (i) of the resolution guarantee is satisfied.

To show part (ii) of the resolution guarantee, 
assume that $D=\emptyset$ and $D_R\neq\emptyset$. Then there must be an 
index ${{\color{myMarkingColor}s}}\in\lbrace 1,2,\ldots,N\rbrace$ with 
$$R(\tau_{{\color{myMarkingColor}s}},z_\mathrm{max})+\delta\mathrm{Id}\geq R_\delta.$$
Using Theorem \ref{th:monotonicity_relation} we obtain 
 \setlength{\arraycolsep}{0.0em}
\begin{eqnarray*}
  -2\delta\mathrm{Id}&{}\leq{}& R(\tau_{{\color{myMarkingColor}s}},z_\mathrm{max})-\left(\delta\mathrm{Id}+R_\delta\right)\\
  &{}\leq{}& R(\tau_{{\color{myMarkingColor}s}},z_\mathrm{max})-R(\sigma,z)\\
  &{}\leq{}& R(\tau_{{\color{myMarkingColor}s}},z_\mathrm{max})-R({\sigma_B}_\mathrm{max},z_\mathrm{min}),
 \end{eqnarray*}
 \setlength{\arraycolsep}{5pt}
and thus $\mu\geq -2\delta$.
\end{proof}

Theorem \ref{th:theorem_to_algorithm_nonlinearised} gives a rigorous yet conceptually simple
criterion to check whether a given resolution guarantee is valid or not. Given a 
partition $(\omega_{{\color{myMarkingColor}s}})_{{{\color{myMarkingColor}s}}=1}^N$, and bounds on the background and contact impedance error,
we obtain $\mu$ from calculating 
$R(\tau_{{\color{myMarkingColor}s}},z_\mathrm{max})$ and $R({\sigma_B}_\mathrm{max},z_\mathrm{min})$
by solving the partial differential equations of the complete electrode model.
If this yields a negative value for $\mu$, then the resolution guarantee holds true
up to a measurement error of $\delta<-\mu/2$.

\subsection{Fast {\color{myMarkingColor}linearized} verification of resolution guarantees}\label{subsec:A_fast_linearised_monotonicity_based_verification}

Checking the criterion in Theorem \ref{th:theorem_to_algorithm_nonlinearised} for a partition
with $N$ elements, requires the solution of $N+1$ forward problems. 
A less accurate but considerably faster test can be obtained by 
replacing the monotonicity tests in algorithm \ref{algo:nonlinearised_reconstruction_algorithm}
\begin{equation*}
 R(\tau_{{\color{myMarkingColor}s}},z_\textrm{max})+\delta\mathrm{Id}\geq R_\delta,
\end{equation*}
with their {\color{myMarkingColor}linearized} approximations
\begin{equation}
 R({\sigma_B}_\mathrm{min},z_\mathrm{max})+\lambda R'({\sigma_B}_\mathrm{min},z_\mathrm{max})(\chi_{\omega_{{\color{myMarkingColor}s}}})+\delta\mathrm{Id}\geq R_\delta,
\end{equation}
where $\lambda\in \mathbb{R}$ is a suitably chosen contrast level {\color{myMarkingColor}(as defined in the algorithms \ref{algo:linearised_reconstruction_algorithm} and \ref{algo:linearised_reconstruction_algorithm_less_conductive_localized})}, 
\begin{equation}\label{eq:adjoint_representation_of_jacobian}
R'\left( {\sigma_B}_\mathrm{min},z_\mathrm{max} \right)(\chi_{\omega_{{\color{myMarkingColor}s}}})
= - \left(\int_{\omega_{{\color{myMarkingColor}s}}}\nabla u_i \cdot \nabla u_j\,\mathrm{d}x\right)_{i,j=1}^{L-1},
\end{equation}
and $u_i$ is the solution of the complete electrode model introduced in 
in section~\ref{sec:CEM} with interior conductivity $ {\sigma_B}_\mathrm{min}$ and
contact impedances $z_\mathrm{max}$. One can interpret $R'$ as the Fr\'echet-derivative of the 
measurements with respect to the interior conductivity distribution, cf., e.g., Lionheart \cite{lionheart2003eit}
or Lechleiter and Rieder \cite{Lec08}, but we will not require this in the following.


{\color{myMarkingColor}
\begin{remark}
The matrix $R'\left({\sigma_B}_\mathrm{min},z_\mathrm{max}\right)(\chi_{\omega_s})$ can be expressed in terms of the sensitivity matrix $S$ that is frequently being used in FEM-based EIT solvers (cf., e.g., \cite{choi2013regularizing} for a recent work in the
context of inclusion detection).

Let $(q_r)_{r=1}^p$ be the elements of a FEM discretization of the considered domain $\Omega$.
The sensitivity matrix $S\in\mathbb{R}^{(L-1)^2\times p}$ is given by
\begin{equation}
 S=\begin{pmatrix}
    S_1\\
    \vdots\\
    S_{L-1}
   \end{pmatrix}, 
\end{equation}
with
\begin{equation}
 S_j=\left(S^{[i,r]}_j\right)=\left(- \int_{q_r}\nabla u_i \cdot \nabla u_j\,\mathrm{d}x\right)\in\mathbb{R}^{L-1\times p}.
\end{equation}
If each element $\omega_s$ in the resolution partition is a union of elements $q_r$ 
of the FEM-discretization, then the entries of $R'$ can be obtained from summing up the corresponding entries
of $S$,
\begin{equation}\label{eq:sensitivity_based_representation_of_R_prime}
 R'\left({\sigma_B}_\mathrm{min},z_\mathrm{max}\right)(\chi_{\omega_s})
=\left(\sum_{r:\ q_r\subseteq \omega_s} S^{[i,r]}_j\right)_{i,j=1}^{L-1}.
\end{equation}
\end{remark}
}

To choose the parameter $\lambda$ we require the additional knowledge of a global bound $\sigma_\mathrm{max}$ with
\begin{equation}
\sigma(x)\leq \sigma_\mathrm{max} \quad \forall x\in \Omega.
\end{equation}

\begin{algorithm}\label{algo:linearised_reconstruction_algorithm}
Mark each resolution element $\omega_{{\color{myMarkingColor}s}}$ for which
\begin{equation}\label{eq:test_criterion_algorithm_1_linearised}
T_{{\color{myMarkingColor}s}}+\delta\mathrm{Id}\geq R_\delta,
\end{equation}
where
\begin{equation}
T_{{\color{myMarkingColor}s}}:=R({\sigma_B}_\mathrm{min},z_\mathrm{max})+\lambda R'({\sigma_B}_\mathrm{min},z_\mathrm{max})(\chi_{\omega_{{\color{myMarkingColor}s}}}), 
\end{equation}
\begin{equation}
\lambda:=(c+\epsilon)\frac{{\sigma_B}_\mathrm{min}}{\sigma_\mathrm{max}},\quad {{\color{myMarkingColor}s}}\in\lbrace 1,2,\ldots,N\rbrace.
\end{equation}
Then the reconstruction $D_R$ is given by the union of the marked resolution elements.
\end{algorithm}

\begin{theorem}\label{th:theorem_to_algorithm_linearised}
The reconstruction of algorithm \ref{algo:linearised_reconstruction_algorithm} fulfils 
the resolution guarantee if 
 \begin{equation}\label{eq:verification_criterion_linearised}
  \mu<-2\delta\leq 0
 \end{equation}
 with
 \begin{equation}
  \mu:=\max_{{{\color{myMarkingColor}s}}=1}^N(\min(\mathrm{eig}(T_{{\color{myMarkingColor}s}}-R({\sigma_B}_\mathrm{max},z_\mathrm{min})))).
 \end{equation}
\end{theorem}
\begin{proof}
First, let $\omega_{{\color{myMarkingColor}s}}\subseteq D$. 
Given a vector $w\in \mathbb{R}^{L-1}$, 
let $u_w$ be the inner potential in a body with interior conductivity 
${\sigma_B}_\mathrm{min}$ and contact impedances $z_\mathrm{max}$ 
that results from driving a current of $w_j$ through the $j$-th electrode, respectively.

Theorem \ref{th:monotonicity_inequality} yields that
\setlength{\arraycolsep}{0.0em}
\begin{eqnarray*}
\lefteqn{w^T(R({\sigma_B}_\mathrm{min},z_\mathrm{max})-R(\sigma,z_\mathrm{max}))w}\\
 &{}\geq{}&\int_\Omega\frac{{\sigma_B}_\mathrm{min}}{\sigma}(\sigma-{\sigma_B}_\mathrm{min})|\nabla u_w|^2\mathrm{d}x,
\end{eqnarray*}
\setlength{\arraycolsep}{5pt}
and since $\omega_{{\color{myMarkingColor}s}}\subseteq D$ implies $\sigma-{\sigma_B}_\mathrm{min}\geq (c+\epsilon)\chi_{\omega_{{\color{myMarkingColor}s}}}$, it follows that
\[
R({\sigma_B}_\mathrm{min},z_\mathrm{max})-R(\sigma,z_\mathrm{max}) \geq
- \lambda R'({\sigma_B}_\mathrm{min},z_\mathrm{max})(\chi_{\omega_{{\color{myMarkingColor}s}}})
\] 
Hence, we obtain from theorem \ref{th:monotonicity_relation} that
\setlength{\arraycolsep}{0.0em}
\begin{eqnarray*}
\lefteqn{T_{{\color{myMarkingColor}s}}+\delta\mathrm{Id}}\\
 &{}={} & R({\sigma_B}_\mathrm{min},z_\mathrm{max})+\lambda R'({\sigma_B}_\mathrm{min},z_\mathrm{max})(\chi_{\omega_{{\color{myMarkingColor}s}}}) 
         +\delta\mathrm{Id}\\
 &{}\geq {} & R(\sigma,z_\mathrm{max}) +\delta\mathrm{Id}{}\geq {}  R(\sigma,z) +\delta\mathrm{Id}\\
 &{}\geq {} & R_\delta.
\end{eqnarray*}
\setlength{\arraycolsep}{5pt}
Hence, $\omega_{{\color{myMarkingColor}s}}$ will be marked, which shows that part (i) of the resolution guarantee is satisfied.

The proof of part (ii) of the resolution guarantee 
is completely analogous to the proof of part (ii) in theorem \ref{th:theorem_to_algorithm_nonlinearised}.
\end{proof}

\subsection{Verification for less conductive inclusions}\label{sec:less_conductive_inclusions}

The theory and the results are almost the same in the case that we consider inclusions that are less conductive than the background.
In that case we set
\begin{equation}
 {\sigma_D}_\mathrm{max}:=\sigma_0-c<{\sigma_B}_\mathrm{min}
\end{equation}
and consider the following algorithm.

\begin{algorithm}\label{algo:nonlinearised_reconstruction_algorithm_less_conductive}
Mark each resolution element $\omega_{{\color{myMarkingColor}s}}$ for which
\begin{equation}\label{eq:test_criterion_algorithm_1_nonlinearised_less_conductive}
 R(\tau_{{\color{myMarkingColor}s}},z_\mathrm{min})-\delta\mathrm{Id}\leq R_\delta,
\end{equation}
where
$$\tau_{{\color{myMarkingColor}s}}:={\sigma_B}_\mathrm{max}\chi_{\Omega\setminus{\omega_{{\color{myMarkingColor}s}}}}+{\sigma_{D}}_\mathrm{max}\chi_{\omega_{{\color{myMarkingColor}s}}},\quad {{\color{myMarkingColor}s}}\in\lbrace 1,2,\ldots, N\rbrace.$$
Then the reconstruction $D_R$ is given by the union of the marked resolution elements.
\end{algorithm}

\begin{theorem}\label{th:theorem_to_algorithm_nonlinearised_less_conductive_localized}
 The reconstruction of algorithm \ref{algo:nonlinearised_reconstruction_algorithm_less_conductive} fulfils
 the Resolution guarantee if
\begin{equation}\label{eq:verification_criterion_nonlinearised_localized_extension_less_conductive}
  \mu>2\delta\geq 0
 \end{equation}
 with
 $$\mu:=\min\limits_{{{\color{myMarkingColor}s}}=1}^N\left(\max\left(\mathrm{eig}\left(R(\tau_{{\color{myMarkingColor}s}},z_\mathrm{min})- R({\sigma_{B}}_\mathrm{min},z_\mathrm{max})\right)\right)\right).$$
 \end{theorem}

\begin{proof}
The proof of part (i) of the resolution guarantee is analogous to the proof of part (i) in theorem \ref{th:theorem_to_algorithm_nonlinearised}.
%
 
To show part (ii) of the resolution guarantee, assume that $D=\emptyset$ and $D_R\neq\emptyset$. Then there must be an index
$i\in \lbrace 1,2,\ldots,N\rbrace$ with
$$R(\tau_{{\color{myMarkingColor}s}},z_\mathrm{min})-\delta\mathrm{Id}\leq R_\delta\leq R(\sigma,z)+\delta\mathrm{Id}.$$
Using theorem \ref{th:monotonicity_relation} we obtain
$$R(\tau_{{\color{myMarkingColor}s}},z_\mathrm{min})-2\delta\mathrm{Id}\leq R({\sigma_{B}}_\mathrm{min},z_\mathrm{max}),$$
and thus $\mu\leq 2\delta$.
\end{proof}

\begin{algorithm}\label{algo:linearised_reconstruction_algorithm_less_conductive_localized}
Mark each resolution element $\omega_{{\color{myMarkingColor}s}}$ for which
\begin{equation}\label{eq:test_criterion_algorithm_1_linearised_less_conductive}
 T_{{\color{myMarkingColor}s}}-\delta\mathrm{Id}\leq R_\delta,
\end{equation}
where
\begin{equation}
T_{{\color{myMarkingColor}s}}:=R({\sigma_B}_\mathrm{max},z_\mathrm{min})+R'({\sigma_B}_\mathrm{max},z_\mathrm{min})(\lambda\chi_{\omega_{{\color{myMarkingColor}s}}}), 
\end{equation}
\begin{equation}
 \lambda:=-(c+\epsilon),\quad {{\color{myMarkingColor}s}}\in\lbrace 1,2,\ldots, N\rbrace.
\end{equation}
Then the reconstruction $D_{R}$ is given by the union of the marked resolution elements.
\end{algorithm}

\begin{theorem}\label{th:theorem_to_algorithm_linearised_less_conductive_localized}
 The reconstruction of algorithm \ref{algo:linearised_reconstruction_algorithm_less_conductive_localized} fulfils the resolution guarantee if
 \begin{equation}\label{eq:verification_criterion_linearised_less_conductive_localized}
   {\color{myMarkingColor}\mu}>2\delta\geq 0
 \end{equation}
 with
 \begin{equation}
  {\color{myMarkingColor}\mu}:=\min\limits_{{{\color{myMarkingColor}s}}=1}^N\left(\max\left(\mathrm{eig}\left(T_{{\color{myMarkingColor}s}}- R({\sigma_{B}}_\mathrm{min},z_\mathrm{max})\right)\right)\right).
 \end{equation}
 \end{theorem}

\begin{proof}
First, let $\omega_{{\color{myMarkingColor}s}}\subseteq D$. Given a vector $w\in\mathbb{R}^{L-1}$, let $u_w$ be the inner potential in a body with interior conductivity
${\sigma_B}_\mathrm{max}$ and contact impedances $z_\mathrm{min}$ that results from driving a current of $w_j$ through the $j$-th electrode,
respectively. As in the proof of theorem \ref{th:theorem_to_algorithm_linearised} 
we obtain by applying theorem \ref{th:monotonicity_relation} and \ref{th:monotonicity_inequality}:
\begin{equation*}
 w^T\left(R({\sigma_B}_\mathrm{max},z_\mathrm{min})-\delta\mathrm{Id}-R_\delta\right)w
 \leq
\lambda \int_D|\nabla u_w|^2\,\mathrm{d}x.
\end{equation*}
This yields
$$T_{{\color{myMarkingColor}s}}-\delta\mathrm{Id}\leq R_\delta.$$
Hence, $\omega_{{\color{myMarkingColor}s}}$ will be marked, which shows that part (i) of the resolution guarantee is satisfied.

The proof of part (ii) of the resolution guarantee is completely analogue to the proof of part (ii) in
theorem \ref{th:theorem_to_algorithm_nonlinearised_less_conductive_localized}.
\end{proof}

\section{Numerical results}\label{sec:numerical_results}
The numerical results in this section are generated with MATLAB\textsuperscript{\textregistered} and the commercial FEM-software COMSOL\textsuperscript{\textregistered}.

{\captionsetup{labelfont={myMarkingColor,bf}}
\captionsetup{font={myMarkingColor}}
\color{myMarkingColor}

In all examples, we used the measurement setup explained in figure \ref{fig:resolution_numerics}.
Conductivities and contact impedances are given in Siemens/meter (S/$m$) and Ohmsquaremeter ($\Omega m^2$), respectively. The unit of length is meter ($m$).
Currents and voltages are measured in milliampere ($m$A) and millivolt ($m$V), respectively.

\subsection{Results for academic examples}\label{sec:numerical_phantoms_unit_data}
}

We consider two measurement setups (see fig. \ref{fig:numerical_results_1} and \ref{fig:numerical_results_2}). For both settings, we assume that
the background conductivity is approximately $\sigma_0=1$ and the
contact impedances are approximately $z_0=(1,1,\ldots,1)\in\mathbb{R}^L$. The inclusions conductivity contrast is assumed to be $c=10$.

{\color{myMarkingColor}
\begin{figure}[!ht]%
\centering
\begin{center}
\begin{tikzpicture}[scale=2]

\newcommand{\myL}{36};
\newcommand{\myLdiv}{9};
\newcommand{\myp}{0.5};

\newcommand{\myb}{((8*\myp)/\myL)};
\newcommand{\myh}{((2*(1-\myp))/(\myL/4+1))};

\foreach \x in {1,2,...,\myLdiv}{
\draw (1,{-1+\myh+(\x-1)*(\myb+\myh)}) rectangle (1.03,{-1+\myh+(\x-1)*(\myb+\myh)+\myb});
\fill (1,{-1+\myh+(\x-1)*(\myb+\myh)}) rectangle (1.03,{-1+\myh+(\x-1)*(\myb+\myh)+\myb});}


\foreach \x in {1,2,...,\myLdiv}{
\draw [rotate=90]  (1,{-1+\myh+(\x-1)*(\myb+\myh)}) rectangle (1.03,{-1+\myh+(\x-1)*(\myb+\myh)+\myb});
\fill [rotate=90]  (1,{-1+\myh+(\x-1)*(\myb+\myh)}) rectangle (1.03,{-1+\myh+(\x-1)*(\myb+\myh)+\myb});
}
\foreach \x in {1,2,...,\myLdiv}{
\draw [rotate=180] (1,{-1+\myh+(\x-1)*(\myb+\myh)}) rectangle (1.03,{-1+\myh+(\x-1)*(\myb+\myh)+\myb});
\fill [rotate=180] (1,{-1+\myh+(\x-1)*(\myb+\myh)}) rectangle (1.03,{-1+\myh+(\x-1)*(\myb+\myh)+\myb});
}
\foreach \x in {1,2,...,\myLdiv}{
\draw [rotate=-90] (1,{-1+\myh+(\x-1)*(\myb+\myh)}) rectangle (1.03,{-1+\myh+(\x-1)*(\myb+\myh)+\myb});
\fill [rotate=-90] (1,{-1+\myh+(\x-1)*(\myb+\myh)}) rectangle (1.03,{-1+\myh+(\x-1)*(\myb+\myh)+\myb});
}

\draw (-1,-1) rectangle (1,1);
\draw (-0.8,-0.8) rectangle (0.8,0.8);
\draw (-0.6,-0.6) rectangle (0.6,0.6);
\draw (-0.3,-0.3) rectangle (0.3,0.3);

\foreach \x in {-1,-0.8,...,1}{
\draw (-1,\x) -- (-0.8,\x);
\draw (1,\x) -- (0.8,\x);
\draw (\x,-1) -- (\x,-0.8);
\draw (\x,1) -- (\x,0.8);}


\foreach \x in {-0.8,-0.6,...,0.8}{
\draw (-0.8,\x) -- (-0.6,\x);
\draw (0.8,\x) -- (0.6,\x);
\draw (\x,-0.8) -- (\x,-0.6);
\draw (\x,0.8) -- (\x,0.6);}

\foreach \x in {-0.6,-0.3,...,0.6}{
\draw (-0.6,\x) -- (-0.3,\x);
\draw (0.6,\x) -- (0.3,\x);
\draw (\x,-0.6) -- (\x,-0.3);
\draw (\x,0.6) -- (\x,0.3);}


\end{tikzpicture}
\end{center}
\caption{$\Omega=[-1,1]^2$ and $36$ electrodes are covering $50\%$ of the boundary.
{\color{myMarkingColor}The first electrode $E_1$ is the lowermost one on the right boundary edge and the electrodes are numbered counter-clockwise.}}
\label{fig:numerical_results_1}%
\end{figure}
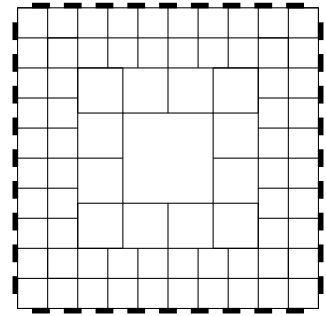

\begin{figure}[!ht]%
\centering
\begin{center}
\begin{tikzpicture}[scale=2]

\newcommand{\myL}{32};
\newcommand{\myLdiv}{8};
\newcommand{\myp}{0.5};

\newcommand{\myb}{((4*\myp)/\myL)};
\newcommand{\myh}{((1*(1-\myp))/(\myL/4+1))};

\foreach \x in {1,2,...,\myLdiv}{
\draw [rotate=-90] (1,{-0.5+\myh+(\x-1)*(\myb+\myh)}) rectangle (1.03,{-0.5+\myh+(\x-1)*(\myb+\myh)+\myb});
\fill [rotate=-90] (1,{-0.5+\myh+(\x-1)*(\myb+\myh)}) rectangle (1.03,{-0.5+\myh+(\x-1)*(\myb+\myh)+\myb});
}

\draw (-1,-1) rectangle (1,1);
\draw (-1,-0) rectangle (1,1);

\foreach \x in {-0.5,0,0.5}{
\draw (\x,-1) -- (\x,0);}
\draw (-1,-0.5) -- (1,-0.5);

\foreach \x in {-0.25,0,0.25}{
\draw (\x,-1) -- (\x,-0.5);}
\draw (-0.5,-0.75) -- (0.5,-0.75);

\foreach \x in {-0.375,-0.25,...,0.375}{
\draw (\x,-1) -- (\x,-0.75);}
\draw (-0.5,-0.875) -- (0.5,-0.875);





\end{tikzpicture}
\end{center}\caption{$\Omega=[-1,1]^2$ and $8$ electrodes are covering $25\%$ of the lower boundary edge.
{\color{myMarkingColor}The electrodes are numbered from the left to the right.}}
\label{fig:numerical_results_2}%
\end{figure}
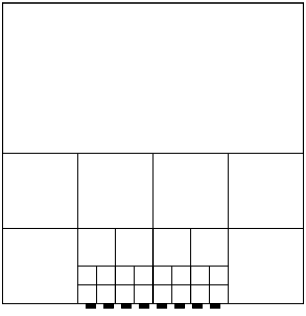
}

The results for figure \ref{fig:numerical_results_1} using our
non-linearised verification procedure in theorem \ref{th:theorem_to_algorithm_nonlinearised} are presented in table \ref{tb:table_1}.
Table \ref{tb:table_2} shows the results for figure \ref{fig:numerical_results_1} obtained from the {\color{myMarkingColor}linearized} verification procedure in theorem \ref{th:theorem_to_algorithm_linearised} under the additional assumption that ${\sigma_D}_\mathrm{max}=15$ is an upper bound on the inclusion contrast. 

\begin{table}[!ht]\caption{%
{\color{myMarkingColor}\bf RG} validation for figure \ref{fig:numerical_results_1} (non-lineari{\color{myMarkingColor}z}ed).}
\begin{center}
\begin{tabular}{|c|c|c|}\hline
background error $\epsilon$:&contact imped.\ error $\gamma$:& abs.\ meas.\ noise $\delta$:\\\hhline{|=|=|=|}
0\%&0\%&0.13\\\hline
0.25\%&0\%&0.11\\\hline
0\%&0.25\%&0.10\\\hline
0.25\%&0.25\%&0.088\\\hline
\end{tabular}
\end{center}\label{tb:table_1}
\end{table}

\begin{table}[!ht]\caption{
{\color{myMarkingColor}\bf RG} validation for figure \ref{fig:numerical_results_1} (lineari{\color{myMarkingColor}z}ed).}
\begin{center}
\begin{tabular}{|c|c|c|c|}\hline
background error $\epsilon$:&contact imped.\ error $\gamma$:& abs.\ meas.\ noise $\delta$:\\\hhline{|=|=|=|}
0\%&0\%&0.051\\\hline
0.25\%&0\%&0.035\\\hline
0\%&0.25\%&0.025\\\hline
0.25\%&0.25\%&0.013\\\hline
\end{tabular}
\end{center}\label{tb:table_2}
\end{table}

The desired resolution shown in the second measurement setup in figure \ref{fig:numerical_results_2} is particularly ambitious. Using the non-linearised verification method it is not possible to guarantee the shown resolution.
Under the additional assumption ${\sigma_D}_\mathrm{max}=12$ on the upper bound of the inclusion contrast, the resolution can be guaranteed using the {\color{myMarkingColor}linearized} validation method up to the errors given in table \ref{tb:table_3}.

\begin{table}[!ht]\caption{
{\color{myMarkingColor}\bf RG} validation for figure \ref{fig:numerical_results_2} (lineari{\color{myMarkingColor}z}ed).}
\begin{center}
\begin{tabular}{|c|c|c|c|}\hline
background error $\epsilon$:&contact imped.\ error $\gamma$:& abs.\ meas.\ noise $\delta$:\\\hhline{|=|=|=|}
0\%&0\%&0.026\\\hline
0.05\%&0\%&0.022\\\hline
0\%&0.05\%&0.0036\\\hline
0.05\%&0.05\%&0.0022\\\hline
\end{tabular}
\end{center}\label{tb:table_3}
\end{table}


{{\captionsetup{labelfont={myMarkingColor}}
\captionsetup{font={myMarkingColor}}
\color{myMarkingColor}
\subsection{Results using physiologically relevant parameters}\label{sec:numerical_phantoms_physiological_data}

The following setting is motivated by the idea of detecting hemorrhages inside fatty tissue. The resolution partition and
the electrodes are concentrated to the lower half of a circle-shaped object $\Omega$.
We used physiological parameter values based on the overview about electric properties of tissue \cite{miklavvcivc2006electric}.
%
%
%
We assume that the background conductivity is approximately $\sigma_0=0.03$.
The inclusion minimal conductivity contrast is $c=0.43-0.03=0.4$ and the upper bound of the inclusion conductivity is ${\sigma_D}_\mathrm{max}=0.7$. 

Since realistic values for contact impedances are typically much smaller than $1$ (cf. \cite{vilhunen2002simultaneous}),
we assume the contact impedance on each electrode to be approximately $0.01$. 

The results for figure \ref{fig:numerical_results_3} using our
non-linearized verification procedure in theorem \ref{th:theorem_to_algorithm_nonlinearised} are presented in table \ref{tb:table_4}.
Table \ref{tb:table_5} shows the results for figure \ref{fig:numerical_results_3} obtained from the linearized verification procedure in theorem \ref{th:theorem_to_algorithm_linearised}. 

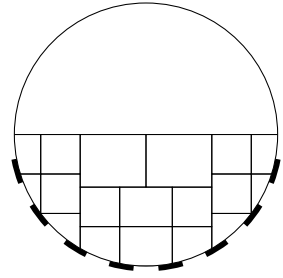
\begin{figure}[!ht]\color{myMarkingColor}%
\centering
\begin{center}
\begin{tikzpicture}[scale=35]

\begin{scope}
\foreach \k in {1,2,...,8}{
  \begin{scope}[rotate=-2*\k*180/17]
  \draw [line width=4pt] (0.05,0) arc (0:1*180/17:0.05);
  \end{scope}
}

\fill[white] (0,0) circle (0.05);

\begin{scope}
\clip (0,0) circle (0.05);

\newcommand{\radius}{0.05};

\newcommand{\lb}{0.3*\radius};

\newcommand{\LLx}{-1.1*\radius};
\newcommand{\LLy}{-3*\lb};
\foreach \j in {1,2}{
    \foreach \k in {1,2,3}{
       \ifnum \j>1
            \draw (\LLx+\j*\lb-\lb,\LLy+\k*\lb-\lb) rectangle (\LLx+\j*\lb,\LLy+\k*\lb);
       \else
	  \ifnum \k>1
	      \draw (\LLx+\j*\lb-\lb,\LLy+\k*\lb-\lb) rectangle (\LLx+\j*\lb,\LLy+\k*\lb);
	  \fi
       \fi
    }
}

\renewcommand{\LLx}{0.5*\radius};
\renewcommand{\LLy}{-3*\lb};
\foreach \j in {1,2}{
    \foreach \k in {1,2,3}{
       \ifnum \j<2
            \draw (\LLx+\j*\lb-\lb,\LLy+\k*\lb-\lb) rectangle (\LLx+\j*\lb,\LLy+\k*\lb);
       \else
	  \ifnum \k>1
	      \draw (\LLx+\j*\lb-\lb,\LLy+\k*\lb-\lb) rectangle (\LLx+\j*\lb,\LLy+\k*\lb);
	  \fi
       \fi
    }
}

\renewcommand{\LLx}{-0.5*\radius};
\renewcommand{\LLy}{-\radius};
\foreach \k in {1,2}{
    \draw (\LLx,\LLy+\k*\lb-\lb) rectangle (\LLx+\lb,\LLy+\k*\lb);
}

\renewcommand{\LLx}{0.2*\radius};
\foreach \k in {1,2}{
    \draw (\LLx,\LLy+\k*\lb-\lb) rectangle (\LLx+\lb,\LLy+\k*\lb);
}

\newcommand{\la}{0.4*\radius};
\renewcommand{\lb}{0.5*\radius};
\renewcommand{\LLx}{-0.5*\radius};
\renewcommand{\LLy}{-0.4*\radius};
\draw (\LLx,\LLy) rectangle (\LLx+\lb,\LLy+\la);

\renewcommand{\LLx}{0*\radius};
\draw (\LLx,\LLy) rectangle (\LLx+\lb,\LLy+\la);

\renewcommand{\lb}{0.3*\radius};
\renewcommand{\LLx}{-0.2*\radius};
\renewcommand{\LLy}{-0.7*\radius};
\draw (\LLx,\LLy) rectangle (\LLx+\la,\LLy+\lb);

\renewcommand{\LLy}{-1*\radius};
\draw (\LLx,\LLy) rectangle (\LLx+\la,\LLy+\lb);

\end{scope}

\draw (0,0) circle (0.05);

\end{scope}

\end{tikzpicture}
\end{center}\caption{$\Omega$ is a disk with diameter of $0.05$ and $8$ electrodes are covering $47\%$ of the lower half of the boundary.
The electrodes are numbered from the left to the right. The resolution partition covers the lower half of the disk.}
\label{fig:numerical_results_3}%
\end{figure}

\begin{table}[!ht]\color{myMarkingColor}%
\caption{{\bf RG} validation for figure \ref{fig:numerical_results_3} (non-linearized)}
\begin{center}
\begin{tabular}{|c|c|c|}\hline
background error $\epsilon$:&contact imped.\ error $\gamma$:& abs.\ meas.\ noise $\delta$:\\\hhline{|=|=|=|}
0\%&0\%&4.4\\\hline
5\%&0\%&0.7\\\hline
0\%&5\%&4.1\\\hline
5\%&5\%&0.6\\\hline
\end{tabular}
\end{center}\label{tb:table_4}
\end{table}

\begin{table}[!ht]\color{myMarkingColor}\caption{{\bf RG} validation for figure \ref{fig:numerical_results_3} (linearized)}
\begin{center}
\begin{tabular}{|c|c|c|}\hline
background error $\epsilon$:&contact imped.\ error $\gamma$:& abs.\ meas.\ noise $\delta$:\\\hhline{|=|=|=|}
0\%&0\%&1.8\\\hline
1\%&0\%&0.7\\\hline
0\%&1\%&1.8\\\hline
1\%&1\%&0.7\\\hline
\end{tabular}
\end{center}\label{tb:table_5}
\end{table}


\subsection{Reconstruction guarantees in a region of interest}\label{sec:blending_of_inhomogeneities}

Our results can be extended to the case where certain areas should be excluded from the region of interest, e.g., if
their background range is known to be violated. As an example, we will add to the setting in section \ref{sec:numerical_phantoms_physiological_data} an area $\omega_I$ consisting of bone and blood beside fat with a conductivity range of $(0.01,0.7)$, cf.\ \cite{miklavvcivc2006electric}.

\vspace*{1em} 

The theory in \ref{subsec:A_monotonicity_based_verification} can be extended as follows: Let
$$\sigma(x)\in\left({\sigma_I}_\mathrm{min},{\sigma_I}_\mathrm{max}\right)\quad\forall x\in\omega_I$$
be the bounds for the conductivity in the area that is to be excluded from the region of interest.
We apply algorithm \ref{algo:nonlinearised_reconstruction_algorithm} with the following changes:
$\tau_s$ in \eqref{eq:def_taus} is replaced by
\begin{equation}
 \tau_{{\color{myMarkingColor}s}}:={\sigma_B}_\mathrm{min}\chi_{\Omega\setminus({\omega_{{\color{myMarkingColor}s}}}\cup\omega_I)}+{\sigma_D}_\mathrm{min}\chi_{\omega_{{\color{myMarkingColor}s}}}+{\sigma_I}_\mathrm{min}\chi_{\omega_I}
\end{equation}
and ${\sigma_B}_\mathrm{max}$ in \eqref{eq:worst_case_test} is replaced by
\begin{equation}
 {\sigma_B}_\mathrm{max}\chi_{\Omega\setminus\omega_I}+{\sigma_I}_\mathrm{max}\chi_{\omega_I}.
\end{equation}
Then, analogously to the result in theorem \ref{th:theorem_to_algorithm_nonlinearised}, we obtain a reconstruction
guarantee where every element covered by the inclusion will be correctly marked, and no element will be marked if
there is no anomaly outside of $\omega_I$.

We tested this variant on the setting shown in figure \ref{fig:numerical_results_4} where $\omega_I$ is assumed to 
consist of bone and blood beside fat with a conductivity range of $(0.01,0.7)$. The results are presented in table \ref{tb:table_6}.

\begin{figure}[!ht]\color{myMarkingColor}%
\centering
\begin{center}
\begin{tikzpicture}[scale=35]

\begin{scope}
\foreach \k in {1,2,...,8}{
  \begin{scope}[rotate=-2*\k*180/17]
  \draw [line width=4pt] (0.05,0) arc (0:1*180/17:0.05);
  \end{scope}
}

\fill[white] (0,0) circle (0.05);

\begin{scope}
\clip (0,0) circle (0.05);

\newcommand{\radius}{0.05};

\newcommand{\lb}{0.3*\radius};

\newcommand{\LLx}{-1.1*\radius};
\newcommand{\LLy}{-3*\lb};
\foreach \j in {1,2}{
    \foreach \k in {1,2,3}{
       \ifnum \j>1
            \draw (\LLx+\j*\lb-\lb,\LLy+\k*\lb-\lb) rectangle (\LLx+\j*\lb,\LLy+\k*\lb);
       \else
	  \ifnum \k>1
	      \draw (\LLx+\j*\lb-\lb,\LLy+\k*\lb-\lb) rectangle (\LLx+\j*\lb,\LLy+\k*\lb);
	  \fi
       \fi
    }
}

\renewcommand{\LLx}{0.5*\radius};
\renewcommand{\LLy}{-3*\lb};
\foreach \j in {1,2}{
    \foreach \k in {1,2,3}{
       \ifnum \j<2
            \draw (\LLx+\j*\lb-\lb,\LLy+\k*\lb-\lb) rectangle (\LLx+\j*\lb,\LLy+\k*\lb);
       \else
	  \ifnum \k>1
	      \draw (\LLx+\j*\lb-\lb,\LLy+\k*\lb-\lb) rectangle (\LLx+\j*\lb,\LLy+\k*\lb);
	  \fi
       \fi
    }
}

\renewcommand{\LLx}{-0.5*\radius};
\renewcommand{\LLy}{-\radius};
\foreach \k in {1,2}{
    \draw (\LLx,\LLy+\k*\lb-\lb) rectangle (\LLx+\lb,\LLy+\k*\lb);
}

\renewcommand{\LLx}{0.2*\radius};
\foreach \k in {1,2}{
    \draw (\LLx,\LLy+\k*\lb-\lb) rectangle (\LLx+\lb,\LLy+\k*\lb);
}

\newcommand{\la}{0.4*\radius};
\renewcommand{\lb}{0.5*\radius};
\renewcommand{\LLx}{-0.5*\radius};
\renewcommand{\LLy}{-0.4*\radius};
\draw (\LLx,\LLy) rectangle (\LLx+\lb,\LLy+\la);

\renewcommand{\LLx}{0*\radius};
\draw (\LLx,\LLy) rectangle (\LLx+\lb,\LLy+\la);

\renewcommand{\lb}{0.3*\radius};
\renewcommand{\LLx}{-0.2*\radius};
\renewcommand{\LLy}{-0.7*\radius};
\draw (\LLx,\LLy) rectangle (\LLx+\la,\LLy+\lb);

\renewcommand{\LLy}{-1*\radius};
\draw (\LLx,\LLy) rectangle (\LLx+\la,\LLy+\lb);

\draw (-0.01,0.02) rectangle (0.01,0.04);
\draw (0,0.03) node {$\omega_I$};

\end{scope}

\draw (0,0) circle (0.05);

\end{scope}

\end{tikzpicture}
\end{center}\caption{$\Omega$ is a disk with diameter of $0.05$ and $8$ electrodes are covering $47\%$ of the lower half of the boundary.
The electrodes are numbered from the left to the right. The resolution partition covers the lower half of the disk. The area $\omega_I$
allows the presence of bone and blood beside fat.}
\label{fig:numerical_results_4}%
\end{figure}
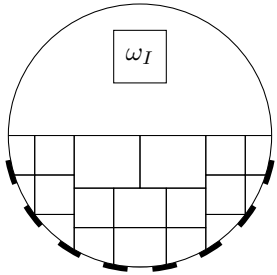

\begin{table}[!ht]\color{myMarkingColor}\caption{{\bf RG} (extended version) validation for figure \ref{fig:numerical_results_4}}
\begin{center}
\begin{tabular}{|c|c|c|}\hline
background error $\epsilon$:&contact imped.\ error $\gamma$:& abs.\ meas.\ noise $\delta$:\\\hhline{|=|=|=|}
0\%&0\%&2.6\\\hline
5\%&0\%&0.3\\\hline
0\%&5\%&2.4\\\hline
5\%&5\%&0.2\\\hline
\end{tabular}
\end{center}\label{tb:table_6}
\end{table}
}}


\section{Conclusion and discussion}\label{sec:conclusion}
We have introduced a rigorous concept of resolution for anomaly detection within realistically {\color{myMarkingColor}modeled} EIT settings. By applying
monotonicity arguments we showed that it is possible to rigorously guarantee
a certain resolution even for settings that include both, systematic {\color{myMarkingColor}modeling} (background and contact impedance) errors and general (e.g., measurement) errors.

We have derived a constructive method to evaluate the amount of errors up to which a given
desired resolution can be guaranteed. We have also derived a linearized variant of our method that allows fast validation of resolution guarantees (while still yielding rigorous results).
In that context let us stress that somewhat surprisingly the linearized variant does not seem to be always inferior to the non-linearized variant as the last example in section \ref{sec:numerical_results} shows.

{\color{myMarkingColor}
Our results may be used to determine whether a desired resolution can be achieved and to quantify
the required measurement accuracy. Moreover, our results could be the basis of optimization strategies regarding the
resolution, or the number and positions of electrodes and the driving patterns.
}

It would be interesting to extend our approach to explicitly address other systematic errors, e.g. including the imaging domain shape and the electrodes position.

{\captionsetup{labelfont={myMarkingColor}}
\captionsetup{font={myMarkingColor}}
\color{myMarkingColor}
\section{Acknowledgements}
The authors would like to thank the German Research Foundation (DFG)
for financial support of the project within the Cluster of Excellence in Simulation
Technology (EXC 310/1) at the University of Stuttgart.
}
\bibliographystyle{IEEEtran}
\bibliography{IEEEabrv,literaturliste}

\begin{thebibliography}{10}
\providecommand{\url}[1]{#1}
\csname url@samestyle\endcsname
\providecommand{\newblock}{\relax}
\providecommand{\bibinfo}[2]{#2}
\providecommand{\BIBentrySTDinterwordspacing}{\spaceskip=0pt\relax}
\providecommand{\BIBentryALTinterwordstretchfactor}{4}
\providecommand{\BIBentryALTinterwordspacing}{\spaceskip=\fontdimen2\font plus
\BIBentryALTinterwordstretchfactor\fontdimen3\font minus
  \fontdimen4\font\relax}
\providecommand{\BIBforeignlanguage}[2]{{%
\expandafter\ifx\csname l@#1\endcsname\relax
\typeout{** WARNING: IEEEtran.bst: No hyphenation pattern has been}%
\typeout{** loaded for the language `#1'. Using the pattern for}%
\typeout{** the default language instead.}%
\else
\language=\csname l@#1\endcsname
\fi
#2}}
\providecommand{\BIBdecl}{\relax}
\BIBdecl

\bibitem{barber1984applied}
D.~Barber and B.~Brown, ``Applied potential tomography,'' \emph{J. Phys. E:
  Sci. Instrum.}, vol.~17, no.~9, pp. 723--733, 1984.

\bibitem{wexler1985impedance}
A.~Wexler, B.~Fry, and M.~Neuman, ``Impedance-computed tomography algorithm and
  system,'' \emph{Applied optics}, vol.~24, no.~23, pp. 3985--3992, 1985.

\bibitem{newell1988electric}
J.~Newell, D.~G. Gisser, and D.~Isaacson, ``An electric current tomograph,''
  \emph{IEEE Trans. Biomed. Eng.}, vol.~35, no.~10, pp. 828--833, 1988.

\bibitem{metherall1996three}
P.~Metherall, D.~Barber, R.~Smallwood, and B.~Brown, ``Three dimensional
  electrical impedance tomography,'' \emph{Nature}, vol. 380, no. 6574, pp.
  509--512, 1996.

\bibitem{cheney1999electrical}
M.~Cheney, D.~Isaacson, and J.~C. Newell, ``Electrical impedance tomography,''
  \emph{SIAM Rev.}, vol.~41, no.~1, pp. 85--101, 1999.

\bibitem{borcea2002electrical}
L.~Borcea, ``Electrical impedance tomography,'' \emph{Inverse problems},
  vol.~18, no.~6, pp. R99--R136, 2002.

\bibitem{borcea2003addendum}
------, ``Addendum to {`}{E}lectrical impedance tomography{'},'' \emph{Inverse
  Problems}, vol.~19, no.~4, pp. 997--998, 2003.

\bibitem{lionheart2003eit}
W.~R.~B. Lionheart, ``\uppercase{EIT} reconstruction algorithms: pitfalls,
  challenges and recent developments,'' \emph{Physiol. Meas.}, vol.~25, pp.
  125--142, 2004.

\bibitem{holder2004electrical}
D.~Holder, \emph{Electrical Impedance Tomography: Methods, History and
  Applications}.\hskip 1em plus 0.5em minus 0.4em\relax Bristol, UK: IOP
  Publishing, 2005.

\bibitem{bayford2006bioimpedance}
R.~Bayford, ``Bioimpedance tomography (electrical impedance tomography),''
  \emph{Annu. Rev. Biomed. Eng.}, vol.~8, pp. 63--91, 2006.

\bibitem{choi2007reconstruction}
M.~H. Choi, T.-J. Kao, D.~Isaacson, G.~Saulnier, and J.~C. Newell, ``A
  reconstruction algorithm for breast cancer imaging with electrical impedance
  tomography in mammography geometry,'' \emph{IEEE Trans. Biomed. Eng.},
  vol.~54, no.~4, pp. 700--710, 2007.

\bibitem{halter2008broadband}
R.~J. Halter, A.~Hartov, and K.~D. Paulsen, ``A broadband high-frequency
  electrical impedance tomography system for breast imaging,'' \emph{IEEE
  Trans. Med. Imaging}, vol.~55, no.~2, pp. 650--659, 2008.

\bibitem{moura2010dynamic}
F.~S. Moura, J.~C.~C. Aya, A.~T. Fleury, M.~B.~P. Amato, and R.~G. Lima,
  ``Dynamic imaging in electrical impedance tomography of the human chest with
  online transition matrix identification,'' \emph{IEEE Trans. Biomed. Eng.},
  vol.~57, no.~2, pp. 422--431, 2010.

\bibitem{adler2011electrical}
A.~Adler, R.~Gaburro, and W.~Lionheart, ``Electrical impedance tomography,'' in
  \emph{Handbook of Mathematical Methods in Imaging}.\hskip 1em plus 0.5em
  minus 0.4em\relax Springer, 2011, pp. 599--654.

\bibitem{martinsen2011bioimpedance}
O.~G. Martinsen and S.~Grimnes, \emph{Bioimpedance and bioelectricity
  basics}.\hskip 1em plus 0.5em minus 0.4em\relax Academic press, 2011.

\bibitem{calderon1980inverse}
A.~P. Calder\'on, ``On an inverse boundary value problem,'' in \emph{Seminar on
  Numerical Analysis and its Application to Continuum Physics}, W.~H. Meyer and
  M.~A. Raupp, Eds.\hskip 1em plus 0.5em minus 0.4em\relax Rio de Janeiro:
  Brasil. Math. Soc., 1980, pp. 65--73.

\bibitem{calderon2006inverse}
------, ``On an inverse boundary value problem,'' \emph{Comput. Appl. Math.},
  vol.~25, no. 2--3, pp. 133--138, 2006.

\bibitem{kohn1984determining}
R.~Kohn and M.~Vogelius, ``Determining conductivity by boundary measurements,''
  \emph{Comm. Pure Appl. Math.}, vol.~37, no.~3, pp. 289--298, 1984.

\bibitem{kohn1985determining}
R.~V. Kohn and M.~Vogelius, ``Determining conductivity by boundary measurements
  {II}. interior results,'' \emph{Comm. Pure Appl. Math.}, vol.~38, no.~5, pp.
  643--667, 1985.

\bibitem{nachman1996global}
A.~I. Nachman, ``Global uniqueness for a two-dimensional inverse boundary value
  problem,'' \emph{Ann. of Math.}, pp. 71--96, 1996.

\bibitem{astala2006calderon}
K.~Astala and L.~P{\"a}iv{\"a}rinta, ``Calder{\'o}n's inverse conductivity
  problem in the plane,'' \emph{Ann. of Math.}, pp. 265--299, 2006.

\bibitem{uhlmann2008commentary}
G.~Uhlmann, ``Commentary on {C}alder{\'o}n’s paper (29) “{O}n an inverse
  boundary value problem”,'' \emph{Selected papers of Alberto P.
  Calder{\'o}n}, pp. 623--636, 2008.

\bibitem{imanuvilov2011determination}
O.~Y. Imanuvilov, G.~Uhlmann, and M.~Yamamoto, ``Determination of second-order
  elliptic operators in two dimensions from partial {C}auchy data,''
  \emph{Proc. Natl. Acad. Sci. USA}, vol. 108, no.~2, pp. 467--472, 2011.

\bibitem{kenig2013recent}
C.~E. Kenig and M.~Salo, ``Recent progress in the {C}alder{\'o}n problem with
  partial data,'' \emph{Contemp. Math.(to appear)}, 2013.

\bibitem{seagar1984full}
A.~Seagar, T.~Yeo, and R.~Bates, ``Full-wave computed tomography. {P}art 2:
  {R}esolution limits,'' \emph{IEE Proceedings}, vol. 131, no.~8, pp. 616--622,
  1984.

\bibitem{seagar1985full}
A.~Seagar and R.~Bates, ``Full-wave computed tomography. {P}art 4:
  {L}ow-frequency electric current {CT},'' \emph{IEE Proceedings}, vol. 132,
  no.~7, pp. 455--466, 1985.

\bibitem{isaacson1986distinguishability}
D.~Isaacson, ``Distinguishability of conductivities by electric current
  computed tomography,'' \emph{IEEE Trans. Med. Imaging}, vol.~5, no.~2, pp.
  91--95, 1986.

\bibitem{gisser1987current}
D.~Gisser, D.~Isaacson, and J.~Newell, ``Current topics in impedance imaging,''
  \emph{Clin. Phys. Physiol. Meas.}, vol.~8, no.~4A, p.~39, 1987.

\bibitem{gisser1990electric}
------, ``Electric current computed tomography and eigenvalues,'' \emph{SIAM J.
  Appl. Math.}, vol.~50, no.~6, pp. 1623--1634, 1990.

\bibitem{paulson1993optimal}
K.~Paulson, W.~Lionheart, and M.~Pidcock, ``Optimal experiments in electrical
  impedance tomography,'' \emph{IEEE Trans. Med. Imaging}, vol.~12, no.~4, pp.
  681--686, 1993.

\bibitem{gencer1994electrical}
N.~Gen\c{c}er, M.~Kuzuoglu, and Y.~Ider, ``Electrical impedance tomography
  using induced currents,'' \emph{IEEE Trans. Med. Imaging}, vol.~13, no.~2,
  pp. 338--350, 1994.

\bibitem{kolehmainen2008electrical}
V.~Kolehmainen, M.~Lassas, and P.~Ola, ``Electrical impedance tomography
  problem with inaccurately known boundary and contact impedances,'' \emph{IEEE
  Trans. Med. Imaging}, vol.~27, no.~10, pp. 1404--1414, 2008.

\bibitem{nissinen2011compensation}
A.~Nissinen, V.~Kolehmainen, and J.~P. Kaipio, ``Compensation of modelling
  errors due to unknown domain boundary in electrical impedance tomography,''
  \emph{IEEE Trans. Med. Imaging}, vol.~30, no.~2, pp. 231--242, 2011.

\bibitem{potthast2006survey}
R.~Potthast, ``A survey on sampling and probe methods for inverse problems,''
  \emph{Inverse Problems}, vol.~22, p.~R1, 2006.

\bibitem{Kir98}
A.~Kirsch, ``Characterization of the shape of a scattering obstacle using the
  spectral data of the far field operator,'' \emph{Inverse Problems}, vol.~14,
  pp. 1489--1512, 1998.

\bibitem{Bru00}
M.~Br{\"u}hl and M.~Hanke, ``Numerical implementation of two noniterative
  methods for locating inclusions by impedance tomography,'' \emph{Inverse
  Problems}, vol.~16, pp. 1029--1042, 2000.

\bibitem{Bru01}
M.~Br{\"u}hl, ``Explicit characterization of inclusions in electrical impedance
  tomography,'' \emph{SIAM J. Math. Anal.}, vol.~32, no.~6, pp. 1327--1341,
  2001.

\bibitem{hanke2003recent}
M.~Hanke and M.~Br{\"u}hl, ``Recent progress in electrical impedance
  tomography,'' \emph{Inverse Problems}, vol.~19, no.~6, pp. S65--S90, 2003.

\bibitem{hyvonen2004complete}
N.~Hyv{\"o}nen, ``Complete electrode model of electrical impedance tomography:
  {A}pproximation properties and characterization of inclusions,'' \emph{SIAM
  J. Appl. Math.}, vol.~64, no.~3, pp. 902--931, 2004.

\bibitem{kirsch2005factorization}
A.~Kirsch, ``The factorization method for a class of inverse elliptic
  problems,'' \emph{Math. Nachr.}, vol. 278, no.~3, pp. 258--277, 2005.

\bibitem{gebauer2006factorization}
B.~Gebauer, ``The factorization method for real elliptic problems,'' \emph{Z.
  Anal. Anwend.}, vol.~25, no.~1, pp. 81--102, 2006.

\bibitem{gebauer2007factorization}
B.~Gebauer and N.~Hyv{\"o}nen, ``Factorization method and irregular inclusions
  in electrical impedance tomography,'' \emph{Inverse Problems}, vol.~23,
  no.~5, p. 2159–2170, 2007.

\bibitem{hyvonen2007numerical}
N.~Hyvonen, H.~Hakula, and S.~Pursiainen, ``Numerical implementation of the
  factorization method within the complete electrode model of electrical
  impedance tomography,'' \emph{Inverse Probl. Imaging}, vol.~1, no.~2, pp.
  299--317, 2007.

\bibitem{nachman2007imaging}
A.~I. Nachman, L.~P{\"a}iv{\"a}rinta, and A.~Teiril{\"a}, ``On imaging
  obstacles inside inhomogeneous media,'' \emph{J. Funct. Anal.}, vol. 252,
  no.~2, pp. 490--516, 2007.

\bibitem{gebauer2008localized}
B.~Gebauer, ``Localized potentials in electrical impedance tomography,''
  \emph{Inverse Probl. Imaging}, vol.~2, no.~2, pp. 251--269, 2008.

\bibitem{hanke2008factorization}
M.~Hanke and B.~Schappel, ``The factorization method for electrical impedance
  tomography in the half-space,'' \emph{SIAM J. Appl. Math.}, vol.~68, no.~4,
  pp. 907--924, 2008.

\bibitem{kirsch2008factorization}
A.~Kirsch and N.~Grinberg, \emph{The Factorization Method for Inverse
  Problems}, ser. Oxford Lecture Ser. Math. Appl.\hskip 1em plus 0.5em minus
  0.4em\relax Oxford: Oxford University Press, 2008, vol.~36.

\bibitem{lechleiter2008factorization}
A.~Lechleiter, N.~Hyv{\"o}nen, and H.~Hakula, ``The factorization method
  applied to the complete electrode model of impedance tomography,'' \emph{SIAM
  J. Appl. Math.}, vol.~68, no.~4, pp. 1097--1121, 2008.

\bibitem{hakula2009computation}
H.~Hakula and N.~Hyv{\"o}nen, ``On computation of test dipoles for
  factorization method,'' \emph{BIT}, vol.~49, no.~1, pp. 75--91, 2009.

\bibitem{harrach2009detecting}
B.~Harrach and J.~K. Seo, ``Detecting inclusions in electrical impedance
  tomography without reference measurements,'' \emph{SIAM J. Appl. Math.},
  vol.~69, no.~6, pp. 1662--1681, 2009.

\bibitem{schmitt2009factorization}
S.~Schmitt, ``The factorization method for {EIT} in the case of mixed
  inclusions,'' \emph{Inverse Problems}, vol.~25, no.~6, p. 065012, 2009.

\bibitem{harrach2010factorization}
B.~Harrach, J.~K. Seo, and E.~J. Woo, ``Factorization method and its physical
  justification in frequency-difference electrical impedance tomography,''
  \emph{IEEE Trans. Med. Imaging}, vol.~29, no.~11, pp. 1918--1926, 2010.

\bibitem{schmitt2011factorization}
S.~Schmitt and A.~Kirsch, ``A factorization scheme for determining conductivity
  contrasts in impedance tomography,'' \emph{Inverse Problems}, vol.~27, no.~9,
  p. 095005, 2011.

\bibitem{harrach2013recent}
B.~Harrach, ``Recent progress on the factorization method for electrical
  impedance tomography,'' \emph{Comput. Math. Methods Med.}, vol. 2013, 2013.

\bibitem{harrach2013monotonicity}
B.~Harrach and M.~Ullrich, ``Monotonicity-based shape reconstruction in
  electrical impedance tomography,'' \emph{SIAM J. Math. Anal.}, vol.~45,
  no.~6, pp. 3382--3403, 2013.

\bibitem{Tamburrino02}
A.~Tamburrino and G.~Rubinacci, ``A new non-iterative inversion method for
  electrical resistance tomography,'' \emph{Inverse Problems}, vol.~18, pp.
  1809--1829, 2002.

\bibitem{Tamburrino06}
A.~Tamburrino, ``Monotonicity based imaging methods for elliptic and parabolic
  inverse problems,'' \emph{J. Inverse Ill-Posed Probl.}, vol.~14, pp.
  633--642, 2006.

\bibitem{Som92}
E.~Somersalo, M.~Cheney, and D.~Isaacson, ``Existence and uniqueness for
  electrode models for electric current computed tomography,'' \emph{SIAM J.
  Appl. Math.}, vol.~52, no.~4, pp. 1023--1040, 1992.

\bibitem{ikehata1998size}
M.~Ikehata, ``Size estimation of inclusion,'' \emph{J. Inverse Ill-Posed
  Probl.}, vol.~6, no.~2, pp. 127--140, 1998.

\bibitem{Kan97}
H.~Kang, J.~K. Seo, and D.~Sheen, ``The inverse conductivity problem with one
  measurement: stability and estimation of size,'' \emph{SIAM J. Math. Anal.},
  vol.~28, no.~6, pp. 1389--1405, 1997.

\bibitem{harrach2010exact}
B.~Harrach and J.~K. Seo, ``Exact shape-reconstruction by one-step
  linearization in electrical impedance tomography,'' \emph{SIAM J. Appl.
  Math.}, vol.~42, no.~4, pp. 1505--1518, 2010.

\bibitem{Lec08}
A.~Lechleiter and A.~Rieder, ``Newton regularizations for impedance tomography:
  convergence by local injectivity,'' \emph{Inverse Problems}, vol.~24, p.
  065009 (18pp), 2008.

\bibitem{choi2013regularizing}
M.~K. Choi, B.~Harrach, and J.~K. Seo, ``Regularizing a linearized {EIT}
  reconstruction method using a sensitivity-based factorization method,''
  \emph{Inverse Problems in Science and Engineering}, no. ahead-of-print, pp.
  1--16, 2013.

\bibitem{miklavvcivc2006electric}
D.~Miklav{\v{c}}i{\v{c}}, N.~Pav{\v{s}}elj, and F.~X. Hart, ``Electric
  properties of tissues,'' \emph{Wiley encyclopedia of biomedical engineering},
  2006.

\bibitem{vilhunen2002simultaneous}
T.~Vilhunen, J.~Kaipio, P.~Vauhkonen, T.~Savolainen, and M.~Vauhkonen,
  ``Simultaneous reconstruction of electrode contact impedances and internal
  electrical properties: {I}. {T}heory,'' \emph{Meas. Sci. Technol.}, vol.~13,
  no.~12, pp. 1848--1854, 2002.

\end{thebibliography}

\end{document}